\documentclass[12pt,a4]{amsart}
\usepackage[a4paper, left=28mm, right=28mm, top=26mm, bottom=31mm]{geometry}
\usepackage{amsthm,amssymb}
\usepackage{cleveref}
\usepackage[all]{xy}
\usepackage{amsmath}
\usepackage{textcomp}
\usepackage{cases}
\usepackage{here}
\usepackage[dvipdfmx]{graphicx}
\usepackage{amscd}
\usepackage{mathrsfs}
\usepackage{float}
\usepackage{wrapfig}
\usepackage{color}

\newtheorem{thm}{Theorem}[section]
\newtheorem{dfn}[thm]{Definition}
\newtheorem{lem}[thm]{Lemma}
\newtheorem{prp}[thm]{Proposition}
\newtheorem{rem}[thm]{Remark}
\newtheorem{cor}[thm]{Corollary}

\crefname{thm}{Theorem}{Theorems}
\crefname{dfn}{Definition}{Definitions}
\crefname{lem}{Lemma}{Lemmas}
\crefname{prp}{Proposition}{Propositions}
\crefname{rem}{Remark}{Remarks}
\crefname{cor}{Corollary}{Corollaries}
\crefname{cond}{Condition}{Conditions}
\crefname{asm}{Assumption}{Assumptions}
\crefname{section}{Section}{Sections}
\crefname{subsection}{Subsection}{Subsections}

\def\P{\mathbb{P}}
\def\A{\mathbb{A}}

\def\G{\mathbb{G}}
\def\R{\mathbb{R}}
\def\Z{\mathbb{Z}}
\def\Q{\mathbb{Q}}
\def\C{\mathbb{C}}
\def\F{\mathbb{F}}
\def\O{\mathcal{O}}

\def\K{\mathscr{K}}
\def\D{\mathscr{D}}
\def\sA{\mathscr{A}}

\def\tropcyc{\mathop{\mathrm{tropcyc}}\nolimits}

\def\wtTrop{\mathop{\mathrm{wtTrop}}\nolimits
\def\Gr{\mathop{\mathrm{Gr}}\nolimits}}
\def\sem{\mathop{\mathrm{sem}}\nolimits}

\def\Aff{\mathop{\mathrm{Aff}}\nolimits}
\def\Dol{\mathop{\mathrm{Dol}}\nolimits}
\def\ch{\mathop{\mathrm{ch}}\nolimits}
\def\cyc{\mathop{\mathrm{cyc}}\nolimits}

\def\alg{\mathop{\mathrm{alg}}\nolimits}
\def\hol{\mathop{\mathrm{hol}}\nolimits}
\def\Kah{\mathop{\mathrm{Kah}}\nolimits}

\def\pr{\mathop{\mathrm{pr}}\nolimits}
\def\log{\mathop{\mathrm{log}}\nolimits}

\def\semialg{\mathop{\mathrm{semi-alg}}\nolimits}
\def\Zar{\mathop{\mathrm{Zar}}\nolimits}

\def\Dol{\mathop{\mathrm{Dol}}\nolimits}

\def\pr{\mathop{\mathrm{pr}}\nolimits}

\def\CH{\mathop{\mathrm{CH}}\nolimits}

\def\Gr{\mathop{\mathrm{Gr}}\nolimits}

\def\res{\mathop{\mathrm{res}}\nolimits}
\def\Im{\mathop{\mathrm{Im}}\nolimits}
\def\Ker{\mathop{\mathrm{Ker}}\nolimits}

\def\sing{\mathop{\mathrm{sing}}\nolimits}

\def\Ber{\mathop{\mathrm{Ber}}\nolimits}
\def\relint{\mathop{\mathrm{rel.int}}\nolimits}

\def\Span{\mathop{\mathrm{Span}}\nolimits}

\def\supp{\mathop{\mathrm{supp}}\nolimits}
\def\ZR{\mathop{\mathrm{ZR}}\nolimits}

\def\Spec{\mathop{\mathrm{Spec}}\nolimits}

\def\Hom{\mathop{\mathrm{Hom}}\nolimits}

\def\Trop{\mathop{\mathrm{Trop}}\nolimits}

\def\alg{\mathop{\mathrm{alg}}\nolimits}
\def\Gr{\mathop{\mathrm{Gr}}\nolimits}

\numberwithin{equation}{section}
\numberwithin{figure}{section}

\makeatletter
\DeclareRobustCommand{\genericinterval}[2]{%
 \@ifstar{\genericinterval@star{#1}{#2}}{\genericinterval@nostar{#1}{#2}}}
  \newcommand{\genericinterval@star}[4]{\mathopen{}\mathclose{\left#1#3,#4\right#2}}
   \newcommand{\genericinterval@nostar}[4]{\mathopen{#1}#3,#4\mathclose{#2}}

        \makeatother

\begin{document}
\title[Differential forms and cohomology in tropical and complex geometry]{Differential forms and cohomology in tropical and complex geometry}
\author{Ryota Mikami}
\address{Institute of Mathematics, Academia Sinica, Astronomy-Mathematics Building, No.\ 1, Sec.\ 4, Roosevelt Road, Taipei 10617, Taiwan.}
\email{mikami@gate.sinica.edu.tw}
\subjclass[2020]{Primary 14T90; Secondary 14C30}
\keywords{tropical geometry, cycle class maps, logarithmic differentials, K-theory, semi-algebraic geometry}
\date{\today}

\begin{abstract}
  Ducros, Hrushovski, and Loeser gave maps from  families of archimedean diffrential forms to non-archiemedean (or tropical) ones, which are compatible with integrals on algebraic varieties. 
  In this paper, we introduce slight modifications of their maps for  complex projective varieties which give natural maps from tropical to the usual Dolbeault cohomology. 
  We also show that our maps are compatible with integrals on generic semi-algebraic subsets and those on their weighted tropicalizations. 
  Weighted tropicalizations induce the dual maps of the above maps of Dolbeault cohomology groups under some assumptions.
\end{abstract}

\maketitle

\setcounter{tocdepth}{1}
\tableofcontents
\section{Introduction}
Cohomology groups for ``tropical'' objects were first considered by Gross and Siebert (\cite{GS10}) in their program of mirror symmetry.
They defined cohomology groups of affine manifolds with singularities with extra datum, and showed that they are isomorphic to logarithmic Dolbeault cohomology groups of certain maximally degenerate Calabi-Yau manifolds (see \cite{GS10} for details).
Itenberg, Katzarkov, Mikhalkin, and Zharkov (\cite{IKMZ19}) introduced \textit{tropcial (singular) cohomology} groups of smooth projective tropical varieties,
and proved similar isomorphisms.
(Yamamoto (\cite{Yam21}) proved these two cohomology of ``tropical'' objects were isomorphic in some cases.)

We can also define \textit{tropical Dolbeault cohomology} groups $H_{\Trop,\Dol}^{p,q}$, which is isomorphic to tropical singular cohomology by the tropical de Rham theorem \cite[Theorem 1]{JSS17}.
This is based on Lagerberg's work (\cite{Lag12}) on \textit{superforms}, tropical analogs of usual bigraded differential forms on complex manifolds.
Chambert-Loir and Ducros's work (\cite{CLD12}) enable us to define tropical Dolbeault (or singular) cohomology for algebraic varieties over non-archimedean fields.

Recently, Ducros, Hrushovski, and Loeser \cite{DHL} 
constructed maps from the set of special kinds of one-parameter families of complex differential forms to the set of superforms for algebraic varieties over an extension of the field of complex numbers which has both archimedean and non-trivial non-archimedean valuations. They also proved equality of integrations of these two kinds of differential forms.

In this paper, we give maps from tropical $H_{\Trop,\Dol}^{p,q}(X)$ to usual Dolbeault cohomology groups $H^{p,q}(X(\C))$ using modifications of Ducros, Hrushovski, and Loeser's maps of differential forms 
in a simple setting, i.e., for a single smooth complex projective variety $X$. (We use the trivial non-archimedean valuation.)
We also prove these maps of two kinds of Dolbeault cohomology groups are compatible with Liu's tropical and usual cycle class maps.
For this, we give a 
new expression  of $\Q$-coefficient tropical $K$-theory $K_T^p$, introduced by the author (\cite{M20-1}).
(Remind 
$$H_{\Trop,\Dol}^{p,q}(X)\cong H^{q}(X_{\Zar},\K_T^p) \otimes \R$$
\cite[Theorem 1.3]{M20-2},
where $K_T^p$ is the Zariski sheafification of $K_T^p$.)
\begin{prp}\label{introduction K group another expression}
   Let $p \geq 0$ be an integer, and $L$  a finitely generated extension of $\C$.
 Then there exists an  isomorphism 
 $$ d(-\frac{1}{2\pi i} \log (-) ) \colon K^p_T(L/\C)
 \cong 
 \Q \langle -\frac{1}{2\pi i} \frac{d f_1}{f_1} \wedge \dots \wedge -\frac{1}{2\pi i} \frac{d f_p}{f_p} \rangle_{f_i \in L^{\times}} \subset \Omega^p_{\Kah}(L/\C), $$
 where $\Omega_{\Kah}$ is the K\"{a}hler differential.
\end{prp}

\begin{rem}
\cref{introduction K group another expression} is a tropical analog of the following result of torsion coefficients Milnor $K$-groups $K_M \otimes_\Z \F_l$: 
  for an integer $p \geq 0$ 
  and a field $L$ of characteristic $l >0$,
  we have 
 $$K^p_M(L) \otimes_\Z \F_l
 \cong 
 \F_l \langle \frac{d f_1}{f_1} \wedge \dots \wedge \frac{d f_p}{f_p} \rangle_{f_i \in L^{\times}} \subset \Omega^p_{\Kah}(L/\F_l) $$
  (Bloch-Kato \cite[Theorem 2.1]{BK86} and Gabber, independently).
\end{rem}

\begin{itemize}
  \item 
We put 
$(\mathscr{A}_{\Trop}^{p,*}, d'')$ the complex of sheaves of $(p,*)$-superforms
on  the  Berkovich analytic space $X^{\Ber}$ over $\C$ with the trivial valuation and
  $\Phi \colon X^{\Ber} \to X$ the map taking supports of valuations.
\item We put
$(\mathscr{D}^{p,*} ,\overline{\partial})$ the complex of sheaves of $(p,*)$-currents on the complex manifold $X(\C)$ and
  $\Psi \colon X(\C) \to X$ the natural map.
\end{itemize}
\begin{thm}\label{main theorem}
  There is a morphism
  $$\Trop^* \colon (\Phi_*\mathscr{A}^{p,*}_{\Trop},d'')\to (\Psi_*\D^{p,*}, \overline{\partial} )$$
  of complexes of sheaves on the Zariski site $X_{\Zar}$ 
extending the morphism
$$ d(-\frac{1}{2\pi i} \log (-) ) \colon \K^p_T 
\to \Omega^p_{\alg} $$
of sheaves on the Zariski site $X_{\Zar}$.
In particular,
  the diagram 
  $$\xymatrix{
	\CH^p(X) \otimes \Q \ar[r]^-{\tropcyc} \ar[dr]^-{\cyc} & H_{\Trop,\Dol}^{p,p}(X) \ar[d]^-{\Trop^*} \\
	 & H^{p,p}(X(\C)) ,
  }$$
  is commutative,
  where 
  $ \Trop^* $
  is the induced map,
  the map 
  $\cyc $
  is the usual cycle class map,
  and  
  $\tropcyc $ 
  is Liu's tropical cycle class map (\cite[Definition 3.6]{Liu20}).
\end{thm}

\begin{rem}
  For a superform $w$, the current $\Trop^*(w)$ is given as the weak limit  of differential forms
 $-\epsilon \log \lvert \cdot \rvert^*(w)$, see \cref{section analytic} for details. 
\end{rem}
\begin{rem}
In the case of toric varieties,
Burgos Gil, Gubler, Jell, and K\"{u}nnemann \cite{BGJK20} introduced a natural map of the similar kind.
They \cite{BGJK21} used it to compare complex and non-archimedean pluripotential theory on toric varieties.
\end{rem}

We also prove equalities of integrations similar to Ducros, Hrushovski, and Loeser's one (\cite{DHL}), but on generic (precisely, admissble) semi-algebraic subsets instead of complex algebraic subvarieties. 
For this purpose, 
we will generalize weighted tropicalizations $\wtTrop(V)$ of algebraic subvarieties to generic semi-algebraic subsets $V$.

\begin{thm}\label{intro comparison of two kinds of integrals}
  Let $\varphi \colon X \to T_{\Sigma}$ be a closed immersion to a smooth projective toric variety $T_{\Sigma}$.
  Let $V \in C_{p+q}(\varphi(X),\C)$ be an admissible singular semi-algebraic chain (\cref{definition of admissible semi-algebraic subsets})
 and $w \in \mathscr{A}_{\Trop}^{p,q}(\Trop(X))$ a superform.
 Then 
 we have 
 $$ \int_{\wtTrop ( V )} w =
 \lim_{\epsilon \to 0} 
 \int_V 
 -\epsilon \log \lvert \cdot \rvert^*(w). 
 $$
\end{thm}

\begin{cor}\label{intro comparison of Trop^* and wtTrop2}
 We assume that 
 $X \subset T_{\Sigma}$ is Sch\"{o}n and 
 every class in the tropical Dolbeault cohomology group $H_{\Trop}^{p,q}(\Trop(X))$ has a $d'$-closed representative.
Then weighted tropicalizations induces a map
$$\wtTrop  \colon 
H^{\sing}_{r}(X,\C) \to 
\bigoplus_{p+q=r} H^{\Trop}_{p,q}(\Trop(X),\C) $$
of homology groups,
and it is the dual map of 
$$\Trop^*  \colon 
\bigoplus_{p+q=r} H_{\Trop,\Dol}^{p,q}(\Trop(X))\otimes_\R \C \to H_{\sing}^{r}(X,\C).$$
\end{cor}
\begin{rem}
 The first assumption in \cref{intro comparison of Trop^* and wtTrop2} is needed to ensure the existence of many admissble semi-algebaic subsets. This assumption is not restrictive because of the existence of Sch\"{o}n open subvarieties \cite[Theorem 1.4]{LQ11}.
 The second assumption is natural as tropical Dolbeault cohomology is analog of the usual Dolbeault cohomology of compact K\"{a}hler manifolds, although the correctness is still not known for most cases.
\end{rem}

In study of mirror symmetry, cycles of complex algebraic varieties constructed from tropical ones are considered, see
\cite{RS20}, \cite{Rud20},  and \cite{RZ}.

Ducros, Hrushovski, and Loeser's theory is based on non-standard models.
Our main tools are  Hironaka's resolution singularities, usual limit arguments of Lebesgue integration, 
and Hanamura, Kimura, and Terasoma's study (\cite{HKT15} and \cite{HKT20}) of logarithmic integrals on admissible semi-algebraic subsets.

The organization of this paper is as follows.
In Section 2, we fix notation and terminology.
In \cref{sectrop}, we remind tropicalization of algebraic varieties and Sch\"{o}n compactifications.
In \cref{sec;trocoho;troK}, we remind superforms, tropical Dolbeault cohomology, and tropical cohomology.
In \cref{section algebraic},
we remind tropical $K$-groups, 
give another expression of them over $\C$ (\cref{introduction K group another expression}), 
and show that this expression
induces maps from tropical cohomology to singular cohomology
which are compatible with  Liu's tropical and the usual cycle class maps.
In \cref{section analytic},
we will construct the map
  $$\Trop^* \colon \mathscr{A}^{p,*}_{\Trop}(X)\to \D^{p,*}(X(\C)).$$
To show that it is well-defined, we will use \cref{introduction K group another expression} and a standard limit argument of the Lebesgue integration.
In \cref{section admissible and logarithmic integrals}, 
we remind Hanamura, Kimura, Terasoma's works on admissble semi-algebaic subsets and logarithmic integrals on them (\cite{HKT15} and \cite{HKT20}).
In \cref{Section Real semi-algebraic construction},
based on Hanamura, Kimura, Terasoma's results,
we construct weighted tropicalizations of admissble semi-algebraic subsets, and prove \cref{intro comparison of two kinds of integrals}.
We also prove that weighted tropicalizations induces maps between $\Q$-coefficient homology groups.

\section{Notation and terminology}
\begin{itemize}
  \item $\Omega_{\alg}^p$: the Zariski sheaf of algebraic differential forms
  \item $\Omega_{\hol}^p$: the sheaf of holomorphic differential forms on complex manifolds
  \item $\Omega_{\Kah}^p$: the K\"{a}hler differential
  \item $\mathscr{A}^{p,q}$: the set of smooth $(p,q)$-forms on complex manifolds
  \item $\mathscr{D}^{p,q}$: the set of $(p,q)$-currents, i.e., continuous linear maps $\mathscr{A}^{p,q} \to \C$
  \item $\mathscr{A}_{\Trop}^{p,q} $: the sheaf of $(p,q)$-superforms on tropical varieties or Berkovich analytic spaces
  \item $H^{p,q}$ : the usual Dolbeault cohomology on complex manifolds
  \item $\partial$: the boundaries  (e.g., of semi-algebraic chains, tropical chains) 
  \item For a finite set $I=(i_1,\dots,i_r) \subset \Z^r $,
  we use multi-index to simplify notations.
  For example,
  \begin{align*}
    x^I:= & x_{i_1} \dots x_{i_r}, \\
  d x_I := & d x_{i_1} \wedge \dots \wedge d x_{i_r}, \\
   d (-\frac{1}{2 \pi i} \log z_I):=& 
   d (-\frac{1}{2 \pi i} \log z_{i_1})\wedge \dots
   \wedge 
   d (-\frac{1}{2 \pi i} \log z_{i_r}).
  \end{align*}
  \item We denote the closure of a subset $A \subset B$ in a topological space $B$ by $\overline{A}^{B}$.
  \item For a rational fan $\Sigma$, we put $T_{\Sigma}$ be the toric variety corresponding to $\Sigma$.
  \item For a $\Z$-module $L$ and an algebra $A$, we put $L_A:= L \otimes_\Z A$.
\end{itemize}

In this paper, fans are polyhedral and cones are strongly convex polyhedral cones.
We assume that fans and cones are rational.

Throughout this paper,
let $M$ be a free $\Z$-module of rank $n$ and  $N:= \Hom_{\Z} (M,\Z)$

\section{Tropicalizations of algebraic varieties}\label{sectrop}
In this section, we 
recall basic properties of fans
(Subsection \ref{subsec;fan}),
\textit{tropicalizations} of complex algebraic varieties 
(Subsection \ref{subsec;trop;Ber}) 
and Sch\"{o}n compactifications (Subsection \ref{subsec;trop;triv;val}).
Let $\Sigma$ be a rational fan in $N_\R$.
For a cone $\sigma \in \Sigma$, we put $N_{\sigma,\R}:= \Hom ( \sigma^{\perp} , \R)$,
where $\sigma^{\perp} \subset M_\R$ is the subset cosisting of elements orthogonal to $\sigma$.

\subsection{Fans}\label{subsec;fan}
In this subsection,
we recall the partial compactification $\bigsqcup_{\sigma \in \Sigma } N_{\sigma ,\R}$ of $\R^n$
and define generalizations of fans in it.

We define a topology on the disjoint union  $\bigsqcup_{\sigma \in \Sigma } N_{\sigma,\R}$  as follows.
We extend the canonical topology on $\R$  to that on $\R \cup \{ \infty \}$ so that  $(a, \infty]$ for $a \in \R$ are a basis of neighborhoods of $\infty$.
We also extend  the addition on $\R$  to that on $\R \cup \{ \infty \}$ by $a + \infty = \infty$ for  $a \in \R \cup \{\infty\}$.
 We consider the set of semigroup homomorphisms $\Hom ( \sigma^{\vee}, \R \cup \{ \infty \}) $  as a topological subspace of $(\R \cup \{ \infty \})^{ \sigma^{\vee}} $.
We define a topology on $\bigsqcup_{ \substack{\tau \in \Sigma \\ \tau \preceq \sigma  } } N_{\tau,\R}$ by the canonical bijection
$$\Hom ( \sigma^{\vee}, \R \cup \{ \infty \}) \cong \bigsqcup_{ \substack{\tau \in \Sigma \\ \tau \preceq \sigma  } } N_{\tau,\R}.$$
 Then we define a topology on  $\bigsqcup_{\sigma \in \Sigma } N_{\sigma,\R}$ by glueing the topological spaces $ \bigsqcup_{\substack{\tau \in \Sigma \\ \tau \preceq \sigma  }} N_{\tau,\R} $ together.

\begin{dfn}
For a cone $\sigma\in \Sigma $ and a   cone $C\subset N_{\sigma,\R}$,
we call its closure $P:=\overline{C}$ in  $ \bigsqcup_{\sigma \in \Sigma } N_{\sigma,\R}$ a  \emph{ cone} in $\bigsqcup_{\sigma \in \Sigma } N_{\sigma,\R}$.
In this case, we put $\relint(P):=\relint(C)$, and call it the \emph{relative interior} of $P$.
We put $\dim(P):=\dim(C)$.
\end{dfn}

Let $\sigma_P \in \Sigma $ be the unique cone such that $\relint(P)\subset N_{\sigma_P,\R}$.
A subset $Q$ of a  cone $P$ in $\bigsqcup_{\sigma \in \Sigma } N_{\sigma,\R}$ is called a \emph{face} of $P$ if it is
the closure of the intersection $P^a \cap  N_{\tau,\R}$
in $ \bigsqcup_{\sigma \in \Sigma } N_{\sigma,\R}$
for some $ a \in \sigma_P\cap M$ and some cone $\tau \in \Sigma $,
where $P^a$ is the closure of
$$\{x \in P\cap  N_{\sigma_P,\R} \mid x(a) \leq  y(a) \text{ for any } y \in P \cap  N_{\sigma_P,\R} \}  $$
in $ \bigsqcup_{\sigma \in \Sigma } N_{\sigma,\R}$.
A finite collection $\Lambda$ of   cones in $ \bigsqcup_{\sigma \in \Sigma } N_{\sigma,\R}$ is called  
a  \emph{fan} if it satisfies the following two conditions:
\begin{itemize}
\item For all $P \in \Lambda$, each face of $P$ is also in $\Lambda$.
\item For all $P,Q \in \Lambda$, the intersection $P \cap Q$ is a face of $P$ and $Q$.
\end {itemize}
We call the union
$$\bigcup_{P\in\Lambda} P \subset \bigsqcup_{\sigma \in \Sigma } N_{\sigma,\R} $$
the \emph{support} of $\Lambda$.
We denote it by $\lvert \Lambda \rvert$.
We also say that $\Lambda$ is a  \emph{fan} \emph{structure} of $\lvert \Lambda \rvert $.

\subsection{Tropicalizations of complex algeraic varieties}\label{subsec;trop;Ber}
We recall basics of \textit{tropicalizations} of complex algebraic varieties.
Let $T_{\Sigma}$ the normal toric variety over $\C$ associated to $\Sigma$. (See \cite{CLS11} for basic notions and results on toric varieties.)
Remind that there is a natural bijection between the cones $\sigma \in \Sigma$ and the torus orbits $O(\sigma)$ in $T_{\Sigma}$.
The torus orbit $O(\sigma)$ is isomorphic to the torus $ \Spec \C[M  \cap \sigma^{\perp}].$
We fix an isomorphism $M \cong \Z^n$, and identify  $\Spec \C[M]$ and $\G_m^n$.

Let $X \subset T_{\Sigma}$ be a closed algebraic variety.
Let $\epsilon >0$ and
$$-\epsilon \log \lvert \cdot \rvert \colon O(\sigma)(\C) \rightarrow N_{\sigma,\R} = \Hom(M \cap \sigma^{\perp} ,\R)$$
the map given by 
$$(x_1,\dots,x_{n-\dim \sigma}) \mapsto (-\epsilon \log \lvert x_1 \rvert, \dots, -\epsilon \log \lvert x_{n-\dim n}\rvert ),$$
here we fix a coordinate of $O(\sigma) \cong \G_m^{n-\dim \sigma}$. (This map is independent of the choice of the coordinate.)
We put 
$$\Trop(X \cap O(\sigma)) := \lim_{\epsilon \to 0} -\epsilon \log \lvert X \rvert \subset N_{\sigma,\R} ,$$
here the limit is the Hausdorff limit by the ususal metric of the Euclid space $N_{\sigma,\R} \cong \R^{n - \dim \sigma}$.
We put 
$$\Trop (X) := \bigsqcup_{\sigma \in \Sigma} \Trop(X \cap O(\sigma)) \subset \bigsqcup_{\sigma \in \Sigma} N_{\sigma,\R} (= \Trop(T_{\Sigma})). $$
We call it the \textit{tropicalization} of $X$.
The tropicalization $\Trop(X)$ is a finite union of $\dim(X)$-dimensional  rational cones by \cite[Theorem A]{BG84}.
There is another construction of $\Trop(X)$ by non-archimedean geometry and the trivival valuation of $\C$; see \cite{Jon16}.

For a toric morphism $\psi \colon T_{\Sigma'} \to T_{\Sigma}$,
there exists a natural morphism $\Trop(T_{\Sigma'}) \to \Trop(T_{\Sigma})$.
We also denote it by $\psi \colon \Trop(T_{\Sigma'}) \to \Trop(T_{\Sigma})$.
The restriction $ \psi|_{N_{\sigma'}} \colon N_{\sigma'} \to N_{\sigma}$ of it to each orbits is a linear map.
When there are closed subvarieties $Y \subset T_{\Sigma'} $ and $X \subset T_{\Sigma}$ such that 
$\psi(Y) \subset X$ (resp. $\psi(Y) = X$),
we have $\psi(\Trop(Y)) \subset \Trop(X)$
(resp. $\psi(\Trop(Y)) = \Trop(X)$).

\subsection{Sch\"{o}n compactifications}\label{subsec;trop;triv;val}
In this subsection, we shall recall   Sch\"{o}n compactifications.
See \cite{Tev07}, \cite{LQ11}, and \cite[Chapter 6]{MS15} for details.
Let $X \subset \G_m^n=\Spec \C[M]$ be a closed subvariety over $\C$.

\begin{dfn}\label{dfn;trop;cpt}
The closure $\overline{X}$ in the smooth toric variety $T_{\Sigma}$ associated with a unimodular fan $\Sigma$ in $N_\R$ is 
called a \emph{Sch\"{o}n compactification}
if the multiplication map
$$ \G_m^n \times \overline{X} \to T_{\Sigma} $$
is smooth and $\overline{X} $ is proper.
\end{dfn}

An affine variety $U$ is called a very affine variety if there exists a closed immersion $U \to \G_m^r$ to a torus $\G_m^r$.
In this case, we call a morphism $U \to \Spec \C[M']$ to a torus $\C[M']$ an intrinsic embedding if the induced morphism $M \to \O(U)^{\times}/ \C^{\times}$ of a lattice is an isomorphism.

Luxton and Qu proved the following  using resolution of singularities.
\begin{thm}[{Luxton-Qu \cite[Theorem 1.4]{LQ11}}]\label{existence of Schon variety}
 There exists an open subvariety $U \subset X$,
 an intrinsic embedding 
 $U \to \G_m^r$,
 and a fan $\Lambda$ in $\R^r$
 such that 
 $\overline{U} \subset T_{\Lambda}$ is Sch\"{o}n compact,
 where $T_{\Lambda}$ is the troic variety associated with a fan $\Lambda$.
\end{thm}

\begin{lem}\label{rem;trop;cpt;property}
Let $\overline{X} \subset T_{\Sigma}$  be a Sch\"{o}n compactification of $X \subset \G_m^n$.
\begin{enumerate}
\item The fan $\Sigma$ is a fan structure of $\Trop(X)\subset N_\R$ \cite[Proposition 2.5]{Tev07}.
\item For any refinement $\Sigma'$ of $\Sigma$, the closure of $X$ in $T_{\Sigma'}$ is also a Sch\"{o}n compactification \cite[Proposition 2.5]{Tev07}.
\item The intersection of $\overline{X}$ and the toric divisors are simple normal crossing divisors of $\overline{X}$.
\end{enumerate}
\end{lem}

\section{Tropical cohomology and superforms}\label{sec;trocoho;troK}
In this section, 
we recall 
\textit{superforms} and \textit{tropical Dolbeault cohomology} (\cref{subsection superforms and tropical Dolbeault cohomology}),
\textit{tropical cohomology}
(Subsection \ref{subsec;trop;cohomology;fan}).

Remind that $M$ is a free $\Z$-module of finite rank and 
we put $N:=\Hom(M,\Z)$.
Let $T_{\Sigma}$ be the toric variety over $\C$ associated with a fan $\Sigma$ in $N_\R$.

\subsection{Superforms and tropical Dolbeault cohomology}\label{subsection superforms and tropical Dolbeault cohomology}
In this subsection,
we recall superforms and tropical Dolbeault cohomology groups.
We start by recalling the definitions of superforms for open subsets of $\R^r$ due to Lagerberg \cite{Lag12}. 

\begin{dfn}
Let $U \subset \R^n$ be an open subset. We denote by $\mathscr{A}^q(U)$ the space of smooth differential forms of degree $q$ on $U$.
The space of \textit{$(p, q)$-superforms} on $U$ is defined as
\begin{align*}
\mathscr{A}_{\Trop}^{p,q}(U) := \mathscr{A}_{\Trop}^p(U) \otimes_{C^{\infty}(U)} \mathscr{A}_{\Trop}^q(U) .
\end{align*}
\end{dfn}
We put 
\begin{align*}
d'x_{i_1} \wedge \dots \wedge d'x_{i_p} \otimes d''x_{j_1} \wedge \dots \wedge d''x_{j_q}
:&=   (dx_{i_1} \wedge \dots \wedge dx_{i_p}) \otimes_\R (dx_{j_1} \wedge \dots \wedge dx_{j_q}) \\
&\in \mathscr{A}_{\Trop}^p(U) \otimes_{C^{\infty}(U)} \mathscr{A}_{\Trop}^q(U), 
\end{align*}
where $x_1,\dots, x_n$ is the coordinate of $\R^n$.
For simplicity, for $I = \{ i_1 , \dots i_p\} $,
we put 
$$d'x_I  :=
d'x_{i_1} \wedge \dots \wedge d'x_{i_p} , \qquad
d''x_I  :=
d''x_{i_1} \wedge \dots \wedge d''x_{i_p} .$$
We put 
\begin{align*}
d'' \colon \mathscr{A}_{\Trop}^{p,q}(U)  \rightarrow \mathscr{A}_{\Trop}^{p,q+1}(U),
\end{align*}
a differential operator
given by $ \text{id} \otimes d$, where $d$ is the usual 
differential operator on the space $\mathscr{A}^*(U)$ of smooth differential forms.  We have
\begin{align*}
d'' \left( \sum_{I,J} \alpha_{I,J} d'x_I \otimes d''x_J \right) =  \sum_{I,J} \sum_{i = 1}^{n} \frac { \partial \alpha_{I,J} } {\partial x_i} d'x_I \otimes d''x_i \wedge d''x_J,
\end{align*} 
where the coefficients $\alpha_{I,J} \in C^\infty(U)$ are smooth functions. 
In the same way, we put 
\begin{align*}
d' \colon \mathscr{A}_{\Trop}^{p+1,q}(U)  \rightarrow \mathscr{A}_{\Trop}^{p,q}(U),
\end{align*}
a differential operator
given by $ (-1)^q d \otimes \text{id} $.

\begin{rem}
Note that sign of our differential $d''$ is  different from  that in \cite{JSS17}. 
We adapt ours so that differential are  compatible with the coboundary map of tropical cochains via integrations,
see \cite[Subsection 5.4]{M20-2}.
\end{rem}

We define  a wedge product of superforms 
\begin{align*}
\wedge \colon \mathscr{A}_{\Trop}^{p,q}(U) \times \mathscr{A}_{\Trop}^{p',q'}(U) &\rightarrow \mathscr{A}_{\Trop}^{p+p',q+q'}(U) \\
(\alpha, \beta) &\mapsto \alpha \wedge \beta,
\end{align*}
  by 
\begin{align*}
(\alpha_{I,J} d'x_I \otimes d''x_J ) \wedge ( \beta_{I',J'} d'x_{I'} \otimes d''x_{J'}) :=(-1)^{pq'}  \alpha_{I,J} \beta_{I',J'} d'x_{I} \wedge d'x_{I'} \otimes d''x_{J} \wedge d''x_{J'}.
\end{align*}

\begin{dfn}\label{def:forms}
Let $U \subset \Trop(T_{\Sigma})$ be an open subset. 
A \textit{$(p,q)$-superform $\alpha$ on $U$} is given by a collection of superforms $(\alpha_\sigma)_{\sigma \in \Sigma}$ 
such that
 \begin{enumerate}
\item    
$$\alpha_\sigma  \in \mathscr{A}_{\Trop}^{p,q}(U \cap \Trop(O(\sigma)))$$ 
for $\sigma \in \Sigma$, 
\item 
for $\sigma \in \Sigma$ and a point $x \in U \cap O(\sigma)$, 
there exists a neighborhood $U_x$ of $x$ in $U$ 
such that 
for a cone $\tau \subset \sigma \ (\tau \in \Sigma)$,
we have 
$$\pi_{\sigma,\tau}(U_{x} \cap \Trop(O(\tau))) \subset U_{x} \cap \Trop(O(\sigma))$$
and  
$$\pi^*_{\sigma,\tau}( \alpha_\sigma |_{U_{x} \cap \Trop(O(\sigma))})|_{U_x \cap \Trop(O(\tau))} = \alpha_\tau|_{U_{x} \cap \Trop(O(\tau))},$$ 
where 
$$\pi_{\sigma,\tau} \colon \Trop (O(\tau)) =N_{\tau,\R} \to N_{\sigma,\R}=\Trop(O(\sigma))$$ 
is the natural map given by the restrictions.
\end{enumerate}
We call the second condition \emph{the boundary condition}.
We denote the set of $(p,q)$-superforms on $U$ by $\mathscr{A}_{\Trop}^{p,q}(U).$

For $\alpha \in \mathscr{A}_{\Trop}^{p,q}(U)$,
we put 
$$d'' \alpha:=(d''\alpha_\sigma)_{\sigma \in \Sigma} \in\mathscr{A}_{\Trop}^{p,q+1}(U).$$
\end{dfn}

\begin{dfn}
Let $C$ be a cone in $\Trop(T_{\Sigma})$ and $\Omega \subset |C|$  an open subset. 
We put $\mathscr{A}_{\Trop}^{p,q}(\Omega)$
the space of the following equivalence classes of $(p,q)$-superforms $\alpha$ which is defined on an open neighborhood $U_{\alpha} \subset \Trop(T_{\Sigma})$ of $\Omega$.
Two superforms are equivalent if and only if 
for any cone $P \in C$,
the restrictions of the two forms 
as elements of $$\mathscr{A}^p (\relint(P))\otimes_{C^{\infty}(\relint(P))} \mathscr{A}^q(\relint(P))$$
are the same,
where we consider $\relint(P)$ as an open subset of the affine span $\Aff(\relint(P)) \cong \R^{\dim P}$.

The differential $d'' \colon 
\mathscr{A}_{\Trop}^{p,q}(\Omega) \to \mathscr{A}_{\Trop}^{p,q+1}(\Omega)$
is well-defined.
\end{dfn}  

\begin{dfn}
  Let $X \subset T_{\Sigma}$ be a complex projective algebraic subvariety.
  We put 
  \begin{align*}
  &H_{\Trop,\Dol}^{p,q}(\Trop(X))\\
  :=&\Ker 
(d'' \colon 
\mathscr{A}_{\Trop}^{p,q}(\Trop(X)) \to \mathscr{A}_{\Trop}^{p,q+1}(\Trop(X)))
/
\Im(d'' \colon 
\mathscr{A}_{\Trop}^{p,q-1}(\Trop(X)) \to \mathscr{A}_{\Trop}^{p,q}(\Trop(X))),
  \end{align*}
called the \emph{tropical Dolbeault cohomology group}.

We put 
  $$H_{\Trop,\Dol}^{p,q}(X):=
  \underrightarrow{\lim}_{\varphi} H_{\Trop,\Dol}^{p,q}(\varphi(\Trop(X))),$$
  where $\varphi \colon X \to T_{\Sigma'}$
  runs through all closed immersions to toric varieties.
  This is also called  the \emph{tropical Dolbeault cohomology group}.
\end{dfn}

When we consider $\C$ as a trivially valued field,
by compactness of the Berkovich analytic space $X^{\Ber}$,
our tropical Dolbeault cohomology coincides with those in \cite{Jel16}, \cite{Liu20}, and \cite{M20-2}.

\subsection{Tropical homology}\label{subsec;trop;cohomology;fan}
We remind tropical (co)homology of fans introduced by Itenberg-Katzarkov-Mikhalkin-Zharkov \cite{IKMZ19}.
See also \cite[Section 3]{JSS17}.

Let $\Lambda $ be a fan in $\Trop(T_{\Sigma})$. 
Remind that for $P \in \Lambda$,
we putted $\sigma_P \in \Sigma$ the cone such that $\relint(P) \subset  N_{\sigma_P,\R}$.
We put
$$\Span(P) :=\Span_\Q (P \cap N_{\sigma_P,\Q} )$$ the $\Q$-linear subspace of $\Trop(O(\sigma_P))$ spanned by $P \cap N_{\sigma_P,\Q}$,
$$F_p(P,\Lambda):= F_p(P, \lvert \Lambda \rvert ):= \sum_{\substack{P' \in \Lambda\cap \Trop(O(\sigma_P))\\ \relint(P) \subset P'}} \wedge^p \Span (P') \subset \wedge^p N_{\sigma_P,\Q}$$
for  a non-negative integer $p \geq 0$,
 and
$$F^p({0},\lvert \Lambda \rvert ):= \wedge^p  M_\Q / \{f \in \wedge^p M_\Q \mid \alpha(f)=0 \ (\alpha \in F_p({0},\lvert \Lambda \rvert ))\} .$$

For $P_1,P_2 \in \Lambda$ with $P_2 \subset P_1$,
there exists a natural  map 
$ i_{P_2 \subset P_1} \colon F_p(P_1) \hookrightarrow F_p(P_2).$

\begin{dfn}\label{deftrocoho}
\begin{enumerate}
\item For every cone $P \in \Lambda$,
we put  $C_q(P)$ the free $\Q$-vector space generated by continuous maps $\gamma \colon \Delta^q \rightarrow B \cap P$ from the standard $q$-simplex $\Delta^q$
such that $$\gamma(\relint(\Delta^q)) \subset  \relint(P)$$
and the image of the relative interior of each face  of $\Delta^q$ is contained in the relative interior of a face of $P$.
We put
$$C_{p,q}^{\Trop}(\Lambda) :=\oplus_{P \in \Lambda} F_p(P,\Lambda) \otimes_\Q C_q( P).$$
We call its elements \emph{tropical $(p,q)$-chains} on $(B,\Lambda)$. \label{enu:def:tro:coho:fan}
\item  For $\gamma \in C_q( P)$, we denote the usual boundary by $\partial(\gamma):=\sum_{i=0}^q (-1)^i \gamma^i$.
For each $v \otimes \gamma \in F_p(P,\Lambda) \otimes C_q( P)$, we put
$$\partial(v \otimes \gamma):= \sum_{i=0}^q (-1)^i i_{P_{\gamma,i} \subset P}(v) \otimes \gamma^i \in C_{p,q-1}(\Lambda),$$
where $P_{\gamma,i} \in \Lambda$ is a face of $P$ such that
$$\gamma^i(\relint(\Delta^{q-1}) )  \subset \relint(P_{\gamma,i}).$$
We obtain complexes $(C_{p,*}^{\Trop}(\Lambda),\partial)$.
\item We define the \emph{tropical homology groups} to be
$$H_{p,q}^{\Trop}(\lvert \Lambda \rvert):=H_{p,q}^{\Trop}(\Lambda):= H_q(C_{p,*}^{\Trop}(\Lambda),\partial).$$
We put $(C^{p,*}_{\Trop} ( \Lambda), \delta)$ the dual complex of $(C_{p,*}^{\Trop}(\Lambda),\partial)$.
We call its cohomology groups
$$H^{p,q}_{\Trop}(\lvert \Lambda \rvert):=H_{\Trop}^{p,q}(\Lambda) := H^q(C^{p,*}_{\Trop}(\Lambda),\delta)$$
the \emph{tropical cohomology groups} of $\Lambda$.
\end{enumerate}
\end{dfn}
We put 
$$C_r^{\Trop}:=\bigoplus_{p+q=r} C_{p,q}^{\Trop}, \quad
H_r^{\Trop}:=\bigoplus_{p+q=r} H_{p,q}^{\Trop} .$$

Tropical (co)homology does not depend on the choice of the fan structure $\Lambda$ of $\lvert \Lambda \rvert$ \cite[Proposition 2.8]{MZ13} (see \cite[Section 3]{JSS17} for proof).

For a smooth quasi-projective algebraic variety $X$ over $\C$,
we put 
  $$H_{\Trop}^{p,q}(X):=
  \underrightarrow{\lim}_{\varphi} H_{\Trop}^{p,q}(\varphi(\Trop(X))),$$
  where $\varphi \colon X \to T_{\Sigma'}$
  runs through all closed immersions to toric varieties.
  This is also called  the \emph{tropical  cohomology group}.

\begin{dfn}\label{dfn;semi;alg;map}
Let $S \subset \R^m$ be a semi-algebraic set.
A map $\varphi \colon  
S \to \Trop(T_\Sigma)$
is said to be \emph{semi-algebraic}
if and only if
the graph $$\Gamma_\varphi \subset \R^{m} \times \Trop(T_\sigma) = \R^m \times N_{\sigma,\R}$$
is semi-algebraic in the usual sense for any cone $\sigma \in \Sigma$.
\end{dfn}

We have a natural isomorphism
$H^{p,q}_{\Trop} (\Lambda) \cong H^{p,q}_{\Trop,\sem-\alg} (\Lambda)$
 \cite[Lemma 5.16]{M20-2}.

\begin{dfn}
  Let $X \subset T_{\Sigma}$ be a smooth projective algebraic varieties.
  Let $\Lambda$ be a fan structure of $\Trop(X)$ and $P\in \Lambda$ a cone.
  Let $v \otimes \gamma \in F_p(P,\Lambda) \otimes C_q( P)$  be a semi-algebraic tropical $(p,q)$-chain (i.e, $\gamma$ is semi-algebraic)
  and 
  $$\alpha d'x_I \otimes d''x_J \in
  \mathscr{A}_{\Trop}^{p,q}(\Trop(X)),$$
  where $\alpha$ is a smooth function and $x_1,\dots,x_n$ is a coordinate of $N_\R $.
  We put 
  $$ \int_{v \otimes \gamma} \alpha_{I,J}d'x_I\otimes d''x_J 
  := \langle d x_I , v \rangle \cdot \int_\gamma \alpha_{I,J} d x_J,$$
  where
  we consider 
  $dx_I \in F^p(\{0\},N_{\sigma_P,\R})$,
  the pairing $\langle d x_I , v \rangle$ is the natural one,
  and 
 $$\int_\gamma \alpha_{I,J} d x_J$$ 
 is the sum of the integrations of the smooth differential forms $\alpha_{I,J} d x_J|_{N_{\tau,\R}}$  $(\sigma_P \subset \tau )$ on semi-algebraic sets in the sense of \cite{HKT15}.

 This is well-defined by the boundary condition \cite[Proposition 5.17]{M20-2}.
\end{dfn}

\begin{prp}[{Jell-Shaw-Smacka \cite[Theorem 1]{JSS17}}]
 Let $X \subset T_{\Sigma}$ be a complex projective algebraic subvariety.
 There is a natural isomorphism 
 $$H_{\Trop}^{p,q}(\Trop(X),\R)
 \cong H_{\Trop,\Dol}^{p,q}(\Trop(X)).$$
 Moreover, this is given by integrations 
 (\cite[Corollary 5.21]{M20-2}).
\end{prp}

\begin{rem}\label{remark weight tropicalization of complex algebraic varieties}
Algebraic subvarieties give  tropical homology classes as follows.
Let $X \subset T_{\Sigma}$ be a closed proper subvariety,
and $Y \subset X$ be a closed subvariety of dimension $p$.
Then, by \cite[Theorem 3.4.14]{MS15} and \cite[Proposition 4.3]{MZ13},
there is a fan structure $\Lambda$ of $\Trop(Y)$
and
the \emph{weight} $m_P \in \Z_{\geq 0}$ for each $P \in \Lambda$ of dimension $p$ satisfying the \emph{balanced condition},
see \cite[Definition 3.4.3]{MS15}. 
For each $P \subset \Lambda$, 
we fix an orientation of $\relint(P)$.
The subvariety $Y$ determines an element 
$$ \wtTrop(Y) := (-1)^{\frac{p(p-1)}{2} }\sum_{P \in \Lambda} m_P 1_{\wedge^q \Span_\Z P} \otimes [P] \in C_{p,p}^{\Trop}(\Trop(X),\Z)/ (\text{subdivisions}),$$
where $[P] \in C_{p}^{\sing}(P)$ is given by the fixed orientiation, see  \cref{definition wedge 1 span P} for the difinition of $1_{\wedge^q \Span_\Z P}$.
This does not depend on the choice of the fan structure $\Lambda$. (Here we do not care the difference between $[P]$ and homeomorphisms $\Delta^{\dim P} \cong P$.)
By the balanced condition, this defines an element of $H_{p,p}^{\Trop}(\Trop(X),\Z).$

When
$\Lambda \subset \Sigma$ 
and
$Y \subset T_{\Lambda}$ is a Sch\"{o}n subvariety, 
then 
the multiplicity $m_P$ coincides with the intersection number $Y \cdot \overline{O(P)}$ \cite[Theorem 6.7.7]{MS15}. 
\end{rem}

\begin{dfn}\label{definition wedge 1 span P}
  For a cone $P \subset \Trop(T_{\Sigma})$ of dimension $q$ and a fixed orientation of $P$,
  we put 
  $$1_{\wedge^q \Span_\Z P} \in \wedge^q \Span_\Z P $$
  the generator as a $\Z$-module
  such that 
  for any positive volume form $\omega $ of $\relint (P)$ and $x \in \relint(P)$,
  we have $\omega_x (1_{\wedge^q \Span_\Z P}) > 0$.
\end{dfn}

\section{Tropical $K$-groups and cycle class maps}\label{section algebraic}
In this section, 
we recall tropical analogs of Milnor $K$-groups, \textit{tropical $K$-groups}, 
give another expression of them over $\C$, 
and show that 
induced maps from tropical cohomology to singular cohomology
are compatible with  Liu's tropical and the usual cycle class maps.

We consider $\C$ as a trivially valued field.
Let $L/\C$ be a finitely generated extension of fields.
Let $p \geq 0$ be a non-negative integer.
We put 
$$  K_T^p(L/\C):=\lim_{\substack{\rightarrow \\ \varphi\colon \Spec L \rightarrow \G_m^r }} F^p(0, \Trop(\overline{\varphi(\Spec L)} )) , $$
where $\varphi \colon \Spec L \to \G_m^r$ runs all $\C$-morphisms to tori  of arbitrary dimensions
and the maps are the pull-back maps 
$$
F^p(0, \Trop(\overline{\varphi(\Spec L)} ) \to 
F^p(0, \Trop(\overline{\varphi'(\Spec L)} )
$$
under toric morphisms $\psi \colon \G_m^{r'} \to \G_m^r$ such that $\psi \cdot \varphi'= \varphi$.
We call it the $p$-th \emph{tropical $K$-group}.
When there is no confusion, we put $K_T^p(L):=K_T^p(L/\C)$.

\begin{prp}[{Mikami \cite[Corollary 5.4]{M20-1}}]\label{valuation K group}
We have
$$K_T^p(L) \cong  \wedge^p ( L^{\times})_\Q / J ,$$
where $J$  is the $\Q$-vector subspace generated by $f \in \wedge^p ( L^{\times})_\Q$
such that $\wedge^p v (f)=0 $ for $v \in \ZR(L/K)$,
where $\ZR(L/K)$ is the space of all non-archimedean valuation of $L$ trivial on $K$, 
and $\wedge^p v  $ is given by considering $v $ as a $\Z$-linear map from the multiplicative group $L^\times$.
\end{prp}

We denote by $\mathscr{K}_T^p$ the sheaf on the Zariski cite $X_{\Zar}$ defined by
$$\mathscr{K}_T^p(U):= \Ker (K_T^p(K(X))\xrightarrow{(\partial_\eta^x)_{ x \in U^{(1)}}} \oplus_{x\in U^{(1)}}K_T^{p-1}(k(x)),$$
where 
$\eta \in X$ is the generic point, 
$U^{(1)}$ is the set of points of an open subset $U \subset X$ of codimension $1$,
and $\partial_\eta^x$ is the residue map, see \cite[Section 5]{M20-1} for the definition of $\partial_\eta^x$.
We have 
$$H^p(X_{\Zar},\K_T^p)\cong \CH^p(X) \otimes \Q $$
\cite[Corollary 1.2]{M20-1}.
By \cite[Theorem 1.3]{M20-2},
we also have 
$$H_{\Trop}^{p,q}(X,\Q) \cong H^q (X_{\Zar}, \K_T^p).$$
The composition 
is 
  Liu's tropical cycle class map 
  $$\tropcyc \colon \CH^r (X) \to H_{\Trop}^{r,r} (X,\R)$$
(\cite[Definition 3.6]{Liu20}) 
  (here we identify tropical cohomology and tropical Dolbeault cohomology by the tropical de Rham theorem \cite[Theorem 1]{JSS17}).

\begin{prp}\label{Fp0 by singular cohomology}
 Let $X \subset \G_m^n=\Spec \mathbb{C}[M]$ be a smooth closed algebraic subvariety over $\mathbb{C}$.
 Then 
 the map 
 $$\wedge^p M \ni f_1 \wedge \dots \wedge f_p \mapsto -\frac{1}{2\pi i} d \log f_1 \wedge \dots \wedge -\frac{1}{2\pi i} d \log f_p \in H_{\sing}^{2p} (\G_m^n(\mathbb{C}),\Q)$$
 induces an isomorphism from 
 $F^p(0,\Trop(X)) $ to
 the $\Q$-vector subspace 
$$ \Q \langle -\frac{1}{2\pi i} d \log f_1  \wedge \dots \wedge -\frac{1}{2\pi i} d \log f_p \rangle_{f_i \in M} \subset \Omega^p_{\hol}(X(\mathbb{C})). $$
\end{prp}

\begin{cor}\label{tropcial K-group another expression}
  Let $L$  be a finitely generated extension of $\C$.
 Then we have 
 $$K^p_T(L/\C)
 \cong 
 \Q \langle -\frac{1}{2\pi i} \frac{d f_1}{f_1} \wedge \dots \wedge -\frac{1}{2\pi i} \frac{d f_p}{f_p} \rangle_{f_i \in L^{\times}} \subset \Omega^p_{\Kah}(L/\C). $$
\end{cor}

\begin{rem}
  \cref{tropcial K-group another expression} is a tropical analog of the following result of Milnor $K$-groups $K_M$.
  Bloch-Kato \cite[Theorem 2.1]{BK86} and Gabber proved independently that 
  for an integer $r \geq 0$ 
  and a field $L$ of characteristic $p >0$,
  we have 
 $$K^r_M(L) \otimes_\Z \F_p
 \cong 
 \F_p \langle \frac{d f_1}{f_1} \wedge \dots \wedge \frac{d f_r}{f_r} \rangle_{f_i \in L^{\times}} \subset \Omega^r_{\Kah}(L/\F_p), $$
 where $K^r_M$ is the $\Z$-coefficient Milnor $K$-group.
\end{rem}

\begin{proof}[Proof of \cref{Fp0 by singular cohomology}]
By \cite[Theorem 1.4]{LQ11},
we may and do assume that 
there exists a fan structure $\Lambda$ of $\Trop (X)$ such that the closure of $X$ in the toric variety $T_\Lambda$ is a Sch\"on compactification.

We put $\overline{X} \subset T_\Lambda$ the closure of $X$ 
and $D:= \overline{X} \setminus X$ a simple normal crossing divisor.
Let $D_i $ be its irreducible components.
Let $D(m) $ the disjoint union of intersections of $m$ irreducible components of $D$.
Let 
$$ f:= \sum_i f_{i,1} \wedge \dots \wedge f_{i,p} \in \wedge^p M$$
$( f_{i,j} \in M)$.
We put 
$$ d(-\frac{ \log f}{2\pi i}) := \sum_i -\frac{1}{2\pi i} d \log f_{i,1} \wedge \dots \wedge -\frac{1}{2\pi i} d \log f_{i,p}  $$
 Then 
 by the valuative description of tropical K-groups(\cref{valuation K group}),
 the image of $f$ in $F^p(0,\Trop(X))$ is $0$
if and only if 
the residue 
$$\res_p ( d(-\frac{ \log f}{2\pi i}) ) \in \Gamma (D(p), \underline{\mathbb{C}}_{D(p)})$$
is $0$, (i.e., for each cone $P \in \Lambda$  of dimension $p$, the restriction of $f$ to $\wedge^p \Span (P)$ is $0$ if and only if the $P$-component of the residue $\res_p ( d(-\frac{ \log f}{2\pi i}) )$ is $0$),
see \cite[3.4.1.3]{CZGT14} for the definition of residue.
By \cite[Lemma 3.4.3 (i)]{CZGT14},  this is equivalent to 
 $$ d(-\frac{ \log f}{2\pi i} ) =0 \in H^0 (\overline{X}, \Gr_p^W \Omega_{\overline{X}}^p(\log D)),$$
 see \cite[3.4.1.2]{CZGT14} for weight filtration.
 Since $\Gr_p^W \Omega_{\overline{X}}^m(\log D)=0$ for $m\leq p-1$,
this is equivalent to 
 $$ d(-\frac{ \log f}{2\pi i} ) =0 \in \mathbb{H}^p (\overline{X}, \Gr_p^W \Omega_{\overline{X}}^*(\log D)).$$
Since weight spectral sequence 
$$ E_{1,W}^{p,q}=\mathbb{H}^{p+q} (\overline{X}, \Gr_{-p}^W \Omega_{\overline{X}}^*(\log D)) \Rightarrow H_{\sing}^p(X(\mathbb{C}),\C)$$
degenerate at $E_2$ page (\cite[Proposition 3.3.21 (iv)]{CZGT14}), and $E_{1,W}^{-p-1,2p} =0$, 
this is equivalent to 
$$ d(-\frac{ \log f}{2\pi i} )=0\in \Gr_{2p}^W H_{\sing}^p(X(\mathbb{C}),\C).$$
Since $H_{\sing }^p (\G_m^n,\Q )$ is of pure weight $2p$, 
this is equivalent to  
$$ d(-\frac{\log f}{2\pi i} )=0 \in H_{\sing}^p(X(\mathbb{C}),\Q).$$
Since Hodge filtration degenerate at $E_1$ (\cite[Proposition 3.3.21 (v)]{CZGT14}), this is equivalent to 
$$ d(-\frac{\log f}{2\pi i} )=0 \in H^0(\overline{X}, \Omega_{\overline{X}}^p(\log D)).$$
\end{proof}

By Riemann extension theorem, we have the following.
\begin{prp}\label{morphism of dual tropicalization shaves}
  Let $X$ be a smooth complex algebraic variety.
  Then the isomorphism of \cref{tropcial K-group another expression} induces a morphism
  $$ d (-\frac{1}{2\pi i} \log (-)  ) \colon \K_T^p \to  \Omega_{\alg}^p$$
  of Zariski sheaves on $X_{\Zar}$.
\end{prp}
\begin{rem}\label{remark of compatibility with the usual and tropical cycle class maps}
  Let $X$ be a smooth complex projective variety.
 The morphism  
  $$ d (-\frac{1}{2\pi i} \log (-)  ) \colon \K_T^p \to  \Omega_{\alg}^p$$
    induces a morphism
  $$ H^q(X_{\Zar},\K_T^p) \to  H^{q}(X_{\Zar},\Omega_{\alg}^p).$$
  Since $H^q(X_{\Zar},\K_T^p) \cong H_{\Trop}^{p,q}(X)$ (\cite[Theorem 1.3]{M20-2}), 
  this induces a morphism
  $$ d (-\frac{1}{2\pi i}\log(-)  ) \colon H_{\Trop}^{p,q}(X) \to  H^{p,q}(X,\mathbb{C}).$$
  This map is compatible with usual and tropical cycle class maps,
  i.e.,
  \begin{align}
   d (-\frac{1}{2\pi i} \log (-) ) \circ \tropcyc =\cyc, 
   \label{compatibility with cycle class maps}
  \end{align}
  where 
  $\cyc \colon \CH^r (X) \to H_{\sing}^{r,r} (X,\Q)$
  is the usual cycle class map, and
  $$\tropcyc \colon \CH^r (X) \to H_{\Trop}^{r,r} (X,\R)$$
  is 
  Liu's tropical cycle class map 
(\cite[Definition 3.6]{Liu20}). 

  This easilly follows from the same assertion 
  but using $H^q(X_{\Zar},\K_M^p)$ instead of  $H^q(X_{\Zar},\K_T^p) \cong H^{p,q}_{\Trop}(X)$. 
  Here, the sheaf $\K_M^p$ is the sheaf of Milnor $K$-groups.
  As far as the author knows, the only proof of this assertion for $H^q(X_{\Zar},\K_M^p)$ in the literature is a proof by Esnault-Parajape \cite{EP} on a webpage.
  For readers convinience, we give their proof here.
\end{rem}

\begin{proof}[Proof of \cref{compatibility with cycle class maps}]
 Since the Chern character $$ \ch \colon K_0^* (X) \otimes \Q \to \CH^*(X) \otimes \Q$$ is an isomorphism (here $K_0^*(X)$ is the Grothendieck group of vector bundles on $X$) of rings,
 it suffices to prove 
 \begin{align}
 d (-\frac{\log}{2\pi i}  ) \circ \tropcyc \circ \ch (V) =\cyc \circ \ch (V) \label{eq of compatibility with cycle class maps chern character}
 \end{align}
 for any vector bundle $V$ on $X$.
 For a line bundle $L$ on $X$,
 by \cite[(2.2.5.2)]{Del71},
 \cref{eq of compatibility with cycle class maps chern character} holds.
 For a vector bundle $V$ on $X$, 
 by induction on the rank of $V$ and pull-back map under $\P(V) \to X$ (where $\P(V)$ is the projective bundle of $V$), 
 \cref{eq of compatibility with cycle class maps chern character} holds.
 (Note that the fact that  the tropical cycle class map is compatible with the intersection product of Chow groups and the product of $H^p(X_{\Zar},\K_T^p)$ follows from an assertion in \cite[Section 14]{Ros96}.)
\end{proof}

\section{Construction of $\Trop^*$}\label{section analytic}
Let $T_\Sigma$ be a smooth quasi-projective toric variety over $\mathbb{C}$
 corresponding to a fan $\Sigma$ in $N_\R$ and $X \subset T_\Sigma$ a smooth closed algebraic subvariety of dimension $d$.
In this section, we shall construct a map
$$\Trop^* \colon \mathscr{A}_{\Trop}^{p,q}(\Trop(X)) \to \mathscr{D}^{p,q}(X(\C)),$$
where $\mathscr{D}^{p,q}(X(\C))$
is the set of $(p,q)$-currents.
Note that by partition of unity, 
every superform $\omega \in \mathscr{A}_{\Trop}^{p,q}(\Trop(X))$ can be lift onto $\Trop(T_\Sigma)$.

\begin{dfn}
  Let  $\epsilon >0$. 
  We define a map 
  \begin{align*}
     &  - \epsilon \log \lvert \cdot \rvert^*  \colon 
     \sA_{\Trop}^{p,q}(\Trop(T_\Sigma)) \to \sA^{p,q}(X(\C))\\
     &  \omega := f(x)d'x_{i_1} \wedge \dots \wedge d'x_{i_p}
     \otimes d'' x_{j_1} \wedge \dots \wedge d''x_{j_q} \\
     &\mapsto 
     f(-\epsilon  \log \lvert z \rvert )
     d ( - \frac{\epsilon}{2} \log  \overline{z_{j_1}}  ) \wedge \dots \wedge 
     d ( - \frac{\epsilon}{2} \log  \overline{z_{j_q}}  ) \wedge \dots \wedge 
     d ( - \frac{1}{2\pi i} \log  z_{i_1}  ) \wedge \dots 
     \wedge d( - \frac{1}{2\pi i}\log  z_{i_p} ).
  \end{align*}
  Here, we denote an element  of $M$ by $z_i$ as a holomorphic function and by $x_i$ as a real-valued function on $\Trop(\G_m^n)=N_\R$,
  and $\overline{z_i}$ is the complex conjugate.
  By the boundary condition of $\omega \in \mathscr{A}_{\Trop}^{p,q}(\Trop(T_{\Sigma}))$ 
  (Definition \ref{def:forms}), 
  this map is well-defined, i.e., the image
  $ - \epsilon \log \lvert \cdot \rvert^*(\omega)$
  is a smooth form.
\end{dfn}

Note that we have 
$$-\epsilon \log \lvert \cdot \rvert^* (d'' \omega )
=\overline{\partial}(-\epsilon \log \lvert \cdot \rvert^* (\omega ))$$

\begin{thm}\label{limit of -epsilon log omega}
 Let $\omega \in \sA_{\Trop}^{p,q}(\Trop(T_{\Sigma}))$ be a superform.
 Then the following holds.
 \begin{enumerate}
   \item The weak limit 
 $$\lim_{\epsilon \to 0}  
       - \epsilon \log \lvert \cdot \rvert^*  (\omega) $$
of $(p,q)$-currents on $X(\C)$ is well-defined,
 i.e., for $\alpha \in \sA_{\mathrm{c}}^{d-p,d-q}(X(\C))$,
 $$\lim_{\epsilon \to 0}  
     \int_{X(\C)}  - \epsilon \log \lvert \cdot \rvert^*  (\omega) \wedge \alpha  $$
    converges to a complex number, and it is continuous in $C^{\infty}$-topology.
  \item The  current
 $$\lim_{\epsilon \to 0}  
       - \epsilon \log \lvert \cdot \rvert^* (\omega) $$
  depends  only on the class of $\omega$ in $\sA_{\Trop}^{p,q}(\Trop (X))$.
 \end{enumerate}
\end{thm}
We put 
 $$ \Trop^* (\omega) := \lim_{\epsilon \to 0}   - \epsilon \log \lvert \cdot \rvert^* (\omega)  \in \D^{p,q}(X(\C)).$$

\begin{proof}
  (1)
  By the existence of Sch\"{o}n very affine open subvarieties \cite[Theorem 1.4]{LQ11}(see \cref{existence of Schon variety} in this paper),
  we may and do assume the followings.
  \begin{itemize}
    \item There is a subfan $\Lambda$ of $\Sigma$ such that $X \subset T_{\Lambda}$ and it is Sch\"{o}n compact.
    \item $\supp(\alpha) \subset T_{\sigma} $ for some $\sigma \in \Lambda$. In the following, we identify $T_{\sigma} \cong \A^s \times \G_m^{n-s}$ $(s := \dim \sigma) $.
    We denote the coordinate of $\A^s \times \G_m^{n-s} (\C)$ by $z_1, \dots z_n$ and that of $\Trop(\A^s \times \G_m^{n-s})$ by $x_1,\dots,x_n$.
    There is a  holomorphic coordinate $w_1,\dots,w_d$ of  some open neighborhood $W \subset X(\C)$ of $\supp(\alpha)$ such that $w_i =z_i$ for $i \leq s$.
    We consider $W \subset \A^d$ by this coordinate.
  \end{itemize}
  
  The integration
  $\int_{X(\C)}  - \epsilon \log \lvert \cdot \rvert^* (\omega) \wedge \alpha  $
  is a sum of the integrations of the form
  \begin{align*}
  \int_{X(\C)} g(w)
  f(-\epsilon \log \lvert z_1\rvert,\dots, -\epsilon \log \lvert z_n \rvert)
  \epsilon^u d \Delta w_1 \wedge d \Delta \overline{w_1} \wedge \dots \wedge d \Delta w_s \wedge d \Delta \overline{w_s} \\
  \wedge d w_{s+1} \wedge d \overline{w_{s+1}} \wedge \dots \wedge  d w_{d} \wedge d \overline{w_d},
  \end{align*}
  where 
  \begin{itemize}
    \item $u\in \Z_{\geq 0}$,
    \item $g(w)$ is a smooth function with compact support $\supp(g) \subset \supp (\alpha)$,
    \item $\Delta (w_i)$ is $w_i$ or $- \log (w_i)$,
    \item $\Delta (\overline{w_i})$ is $\overline{w_i}$ or $- \epsilon \log (\overline{w_i})$,
    \item $f(x_1,\dots,x_n) \in \mathscr{A}_{\Trop}^{0,0}(U\cap \Trop(\A^s \times \G_m^{n-s}))$ is a compactly-supported smooth function 
    such that $f$ is $0$ near $x_i=\infty$ for $1 \leq i \leq s$ satisfying
    $\Delta (w_i)=- \log (w_i)$ or $\Delta (\overline{w_i}) = - \epsilon \log (\overline{w_i})$.
    (This condition comes from the boundary condition of $\omega \in \mathscr{A}_{\Trop}^{p,q}(U)$ (\cref{def:forms}).)
  \end{itemize}

  By using polar coordinates $w_i(r_i,\theta_i) = r_i e^{i \theta_i} \ (1 \leq i \leq s)$,
  the integration
  $\int_{X(\C)}  - \epsilon \log \lvert \cdot \rvert^* (\omega) \wedge \alpha  $
  is a sum of the integration of the form
  \begin{align*}
  & \int_{(r_1,\dots,r_s,\theta_1,\dots,\theta_s,w_{s+1},\dots,w_{d}) 
  \in \R_{>0}^s \times [0,2\pi]^s \times (\C^{\times})^{d-s}} \epsilon^{u'} e^{i\theta} r^J g'(w_1,\dots,w_d)
  f(-\epsilon \log \lvert z_1\rvert,\dots, -\epsilon \log \lvert z_n \rvert) \\
  & d \Delta r_1 \wedge  \dots \wedge d \Delta r_s 
  \wedge d \theta_1 \wedge \dots \wedge d \theta_s
  \wedge d w_{s+1} \wedge d \overline{w_{s+1}} \wedge \dots \wedge  d w_{d} \wedge d \overline{w_d},
  \end{align*}
  where 
  \begin{itemize}
    \item $r^J:= \prod_{j \in J}r_j$ for some $J \subset \{1,\dots,d\}$, 
    \item $\theta$ is a linear sum of $\theta_1,\dots,\theta_s$,
    \item $u'\in \Z_{\geq 0}$,
    \item $g'(w_1,\dots,w_d)$ is a smooth function with compact support $\supp(g') \subset \supp (\alpha)$,
    \item $\Delta (r_i)$ is $r_i$ or $- \epsilon \log r_i$.
  \end{itemize}
  We may  assume $\Delta(r_i) =-\epsilon \log r_i$ for $1 \leq i \leq r$ and $\Delta (r_i)= r_i$ otherwise. We may assume $J \subset \{r+1, \dots,d\}$.
  Remind that $x_i= -\epsilon \log r_i \ (1 \leq i \leq r)$ by the map 
  $$-\epsilon \log \lvert \cdot \rvert  \colon T_{\Sigma}(\C) \to \Trop(T_{\Sigma}).$$
  We also denote by $W$ the image of $W$ under the map 
  $$(\C^{\times})^n \ni (z_i=r_i e^{i \theta_i })_i \to (-\epsilon \log r_i, r_j, \theta_k, z_l) \in \R^r \times \R_{>0}^{s-r} \times [0,2\pi]^s \times (\C^{\times})^{d-s}.$$
  Then the above integration is
  \begin{align*}
  & \int_{(x_1,\dots,x_r,r_{r+1},\dots,r_s,\theta_1,\dots,\theta_s,w_{s+1},\dots,w_{d}) \in \R^r \times \R_{>0}^{s-r} \times [0,2\pi]^s \times (\C^{\times})^{d-s} }
  \epsilon^{u'} e^{i \theta} r^J \\
  & g'(w_1(e^{\frac{-x_1}{\epsilon}},\theta_1), \dots,w_r(e^{\frac{-x_r}{\epsilon}},\theta_r), w_{r+1}(r_{r+1},\theta_{r+1}),\dots,w_{s}(r_{r},\theta_{s}),w_{s+1},\dots,w_d  ) \\
  &
  f(x_1,\dots,x_r,-\epsilon \log \lvert z_{r+1}\rvert,\dots, -\epsilon \log \lvert z_{n} \rvert) \\
  & d x_1 \wedge  \dots \wedge d x_r 
  \wedge d r_{r+1} \wedge  \dots \wedge d r_s 
  \wedge d \theta_1 \wedge \dots \wedge d \theta_s
  \wedge d w_{s+1} \wedge d \overline{w_{s+1}} \wedge \dots \wedge  d w_{d} \wedge d \overline{w_d}. 
  \end{align*}
  By the condition of $f$,
  the set $\pr_{1}(\supp(f))  \subset \R^r$
  is compact, 
  where 
  $$ \pr_1 \colon \R^r \times \R_{>0}^{s-r} \times [0,2\pi]^s \times (\C^{\times})^{d-s} \to \R^r$$
  is the projection to the first $r$ coordinates.
  Since  
  $$ \pr_2 (\supp (g) ) \subset
  \R_{>0}^{s-r} \times [0,2\pi]^s \times (\C^{\times})^{d-s}$$
  is bounded, 
  where
  $$ \pr_2 \colon \R^r \times \R_{>0}^{s-r} \times [0,2\pi]^s \times (\C^{\times})^{d-s} \to 
  \R_{>0}^{s-r} \times [0,2\pi]^s \times (\C^{\times})^{d-s}
  $$
  is the projection, 
  there exists a integrable function $h>0$ 
  on
  $$ 
  \R^r \times \R_{>0}^{s-r} \times [0,2\pi]^s \times (\C^{\times})^{d-s}$$
  such that 
  $\lvert r^J g' f \rvert \leq h$ for any small $\epsilon >0$.
  Hence by a usual limit argument of Lebesgue integration, 
  the limit 
  \begin{align*}
  &\lim_{\epsilon \to 0} \int_{(x_1,\dots,x_r,r_{r+1},\dots,r_s,\theta_1,\dots,\theta_s,w_{s+1},\dots,w_{d}) \in \R^r \times \R_{>0}^{s-r} \times [0,2\pi]^s \times (\C^{\times})^{d-s}}
  \epsilon^{u'} e^{i \theta} r^J \\
  & g'(w_1(e^{\frac{-x_1}{\epsilon}},\theta_1), \dots,w_r(e^{\frac{-x_r}{\epsilon}},\theta_r), w_{r+1}(r_{r+1},\theta_{r+1}),\dots,w_{s}(r_{r},\theta_{s}),w_{s+1},\dots,w_d  ) \\
  &
  f(x_1,\dots,x_r,-\epsilon \log \lvert z_{r+1}\rvert,\dots, -\epsilon \log \lvert z_{n} \rvert) \\
  & d x_1 \wedge  \dots \wedge d x_r 
  \wedge d r_{r+1} \wedge  \dots \wedge d r_s 
  \wedge d \theta_1 \wedge \dots \wedge d \theta_s
  \wedge d w_{s+1} \wedge d \overline{w_{s+1}} \wedge \dots \wedge  d w_{d} \wedge d \overline{w_d}
  \end{align*}
  exists and absolutely converges to a complex number
  \begin{align*}
  & \int_{(x_1,\dots,x_r,r_{r+1},\dots,r_s,\theta_1,\dots,\theta_s,w_{s+1},\dots,w_{d}) \in 
(-\infty,\infty)^r \times \R_{>0}^{s-r} \times [0,2\pi]^s \times (\C^{\times})^{d-s}}
  0^{u'} e^{i \theta} r^J \\
  &g'(0 \dots,0, w_{r+1}(r_{r+1},\theta_{r+1}),\dots,w_{s}(r_{r},\theta_{s}),w_{s+1},\dots,w_d  ) \\
  &
  f(x_1,\dots,x_r,0,\dots,0) \\
  & d x_1 \wedge  \dots \wedge d x_r 
  \wedge d r_{r+1} \wedge  \dots \wedge d r_s 
  \wedge d \theta_1 \wedge \dots \wedge d \theta_s
  \wedge d w_{s+1} \wedge d \overline{w_{s+1}} \wedge \dots \wedge  d w_{d} \wedge d \overline{w_d}.
  \end{align*}

By this expression, the linear function $\Trop^*(w)$ on $ \mathscr{A}^{p,q}(X(\C))$  is continuous in $C^{\infty}$-topology.

(2)
Let 
$$\omega := \sum_{I=(I_1,I_2)} f_I(x) d'x_{I_1}\otimes d'' x_{I_2  }$$
whose restriction to $\Trop (X)$ is $0$.
We shall show $\Trop^* (\omega)= 0$.
We may assume that $\supp (w) \subset T_{\sigma}$
for some $\sigma \in \Lambda$. In the following, we identify $T_{\sigma} \cong \A^s \times \G_m^{n-s}$ $(s := \dim \sigma) $.
We denote the coordinate of $\A^s \times \G_m^{n-s} (\C)$ by $z_1, \dots z_n$ and that of $\Trop(\A^s \times \G_m^{n-s})$ by $x_1,\dots,x_n$.
There is a  holomorphic coordinate $w_1,\dots,w_d$ of  some open neighborhood $W \subset X(\C)$ of $\supp(\alpha)$ such that $w_i =z_i$ for $i \leq s$.

By the proof of (1),
the current $\Trop^*(f_I (x) d' x_{I_1} \wedge d'' x_{I_2})$  is non-zero 
only if 
$I_2 \subset I_1$ and $I_2 \subset \{1,\dots,s\}$.
For each $J \subset \{1,\dots,s\}$,
the restriction of 
$$\omega_J := \sum_{I=(I_1,J), J  \subset I_1 } f_I(x) d'x_{I_1} \otimes d'' x_{J}$$
to $\Trop (X)$ is $0$.
It suffices to show that 
$\Trop^*(\omega_J)=0$.

We may and do assume that 
$J= \{1,\dots,r\}$ for some $r \leq s$.
Then by \cref{Fp0 by singular cohomology},  for any $(a_1,\dots,a_r) \in [0,\infty)^r $,
the differential form 
$$\sum_{I=(I_1,\{1,\dots,r\}), \{1,\dots,r\} \subset I_1} 
f_I(a_1,\dots,a_r,0,\dots,0) 
d (- \frac{1}{2\pi i}\log z_{I_1 - \{1,\dots,r\}})$$
is $0$ on $X \cap O(\sigma_r)(\C)$.
Hence for $\alpha = g d w_K \wedge d \overline{w}_L\in \mathscr{A}^{d-p,d-q}(X(\C))$,
in the same way as the proof of (1),
we have
\begin{align*}
&\Trop^*(\omega_J)(\alpha) \\
=& \lim_{\epsilon \to 0} \int_{(x_i,\theta_i,w_k) \in
(-\infty,\infty)^r  \times [0,2\pi]^r \times (\C^{\times})^{d-r}}
 g(w_1 (e^{\frac{-x_1}{\epsilon}},\theta_1),w_r(e^{\frac{-x_r}{\epsilon}},\theta_r) ,w_{r+1},\dots,w_n) \\
&\sum_{I=(I_1,\{1,\dots,r\}), \{1,\dots,r\} \subset I_1} f_I (x_1,\dots,x_r,-\epsilon \lvert z_{r+1}\rvert ,\dots,-\epsilon \lvert z_{n}\rvert ) 
d x_i \wedge d (-\frac{1}{2 \pi}\theta_i) \\
&\wedge 
d (- \frac{1}{2\pi i}\log z_{I_1 - \{1,\dots,r\}})
\wedge d w_K \wedge d \overline{w}_L  \\
=& \int_{(x_i,\theta_i,w_k) \in
(-\infty,\infty)^r  \times [0,2\pi]^r \times (\C^{\times})^{d-r}}
   g(0,\dots,0,w_{r+1},\dots,w_n) \\
&\sum_{I=(I_1,\{1,\dots,r\}), \{1,\dots,r\} \subset I_1} f_I (x_1,\dots,x_r,0,\dots,0) 
d x_i \wedge d (-\frac{1}{2\pi} \theta_i) \wedge 
d (- \frac{1}{2\pi i}\log z_{I_1 - \{1,\dots,r\}})  
\wedge d w_K \wedge d \overline{w}_L \\
=&0.
\end{align*}
\end{proof}

In the proof of \cref{limit of -epsilon log omega}, we have shown the following.
\begin{prp}\label{image of Trop* explicit}
  We assume that there is a subfan $\Lambda \subset \Sigma$ 
  such that $X \subset T_{\Lambda}$ is a Sch\"{o}n compactification.
  Then for $p \geq q$,
  the image of 
  $$\Trop^* \colon \mathscr{A}_{\Trop}^{p,q} (\Trop(X)) \to \mathscr{D}^{p,q}(X(\C))$$
  is the $\R$-vector space 
  spaned  by elements 
  $$\bigg(\alpha \mapsto \int_{X(\C) \cap O(\sigma)(\C)} 
  d(-\frac{1}{2 \pi i} \log z_1) \wedge  \dots \wedge
  d(-\frac{1}{2 \pi i} \log z_{p-q}) \wedge  \alpha \bigg)$$
  $(\sigma \in \Lambda$ of dimension $q$ and $z_i \in M \cap \sigma^{\perp} )$, 
  and
  for $p \leq q-1$, the map $\Trop^*$ is the $0$-map.
\end{prp}
\begin{rem}
  For $\omega \in \mathscr{A}_{\Trop}^{p,q-1}(\Trop(X))$ and $\alpha \in \mathscr{A}_{\mathrm{c}}^{d-p,d-q}(X(\C))$,
we have 
\begin{align*}
& \overline{\partial}(\Trop^*(\omega))(\alpha)\\
= (-1)^{p+q} &\Trop^*(\omega)(\overline{\partial} \alpha)\\
=&\lim_{\epsilon \to 0}
 (-1)^{p+q} \int_{X(\C)} -\epsilon \log \lvert \cdot \rvert^* (\omega) \wedge \overline{\partial}\alpha \\
=&\lim_{\epsilon \to 0} \int_{X(\C)} \overline{\partial} ( -\epsilon \log \lvert \cdot \rvert^* (\omega)) \wedge \alpha \\
=&\lim_{\epsilon \to 0} \int_{X(\C)} -\epsilon \log \lvert \cdot \rvert^* (d'' \omega) \wedge \alpha \\
=& (\Trop^*(d''\omega))(\alpha).
\end{align*}

In the rest of this section, we assume that $X$ is projective.
The map
$$\Trop^* \colon \mathscr{A}_{\Trop}^{p,*}(\Trop(X)) \to \D^{p,*}(X(\C))$$
induces a map of the tropical and the usual Dolbeault cohomology groups
$$\Trop^* \colon H_{\Trop,\Dol}^{p,q}(\Trop(X)) \to H^{p,q}(X(\C)).$$
For a smooth projective variety $X$, 
by the family of closed immersions of $X$ to toric varieties,
we get  a map
$$\Trop^* \colon H_{\Trop,\Dol}^{p,q}(X)\to H^{p,q}(X(\C)).$$
\end{rem}

\begin{rem}\label{maps of cohomology are the same}
  We have 
$$ d(-\frac{1}{2\pi i} \log (-) )= \Trop^*  \colon
H_{\Trop,\Dol}^{p,q}(X) \to H^{p,q}(X(\C)).$$
  This can be seen as follows.

\begin{itemize}
  \item There is a sheaf $\mathscr{F}^p \otimes_\Q \R$ (see \cite[Section 5.2]{M20-2})
on  the  Berkovich analytic space $X^{\Ber}$ over $\C$ with the trivial valuation.
  The tropical Dolbeault cohomology $H_{\Trop,\Dol}^{p,q}(X)$ is its sheaf cohomology, and 
  the complex of the sheaves $\sA_{\Trop}^{p,*}$ of $(p,*)$-superforms (see \cite[Section 3]{Jel16} for definition) is a fine (and hence acyclic) resolution of $\mathscr{F}^p \otimes_\Q \R$.
  \item 
  Let $\Phi \colon X^{\Ber} \to X_{\Zar}$ is the map taking supports of valuations.
  Then the complex $\Phi_* \sA_{\Trop}^{p,*}$ is an acyclic resolution of $\Phi_* \mathscr{F}^p \otimes_\Q \R \cong \K_T^p \otimes_\Q \R$, which follows from \cite[Remark 6.3]{M20-2}.
  (Strictly speaking, \cite[Remark 6.3]{M20-2} says that $\Phi_* \mathscr{C}^{p,*}$ is exact, where $\mathscr{C}^{p,*}$ is the complex of sheaves of tropical cochains on $X^{\Ber}$. By \cite[Theorem 1]{JSS17}, this is equivalent to exactness of $\Phi_* \sA_{\Trop}^{p,*}$.)
  \item
  Our maps 
  $$\Trop^* \colon \mathscr{A}_{\Trop}^{p,q}(\Trop(\varphi(V))) \to \D^{p,q}(V)$$ 
  for open subvarieties $V \subset X$ and closed immersions $\varphi \colon X \to T_{\Sigma'}$ to toric varieties
  give a map 
  $$\Phi_* \mathscr{A}_{\Trop}^{p,*} \to \Psi_* \D^{p,*}$$
  of complex of sheaves on the Zarisiki topology $X_{\Zar}$, 
  where 
  the map $\Psi \colon X(\C) \to X$ is the natural one.
 
  The map 
$$ d(-\frac{1}{2\pi i} \log (-) ) \colon
H_{\Trop,\Dol}^{p,q}(X) \to H^{p,q}(X(\C))$$
is given by this morphism of complex of sheaves.
\end{itemize}

\end{rem}

\section{Semi-algebraic geometry and logarithmic integrals}\label{section admissible and logarithmic integrals}
In this section, we recall admissible semi-algebraic subsets, introduced by Hanamura, Kimura, and Terasoma, and logarithmic integral on them. See \cite{HKT15} and \cite{HKT20} for details.
These will be used in \cref{Section Real semi-algebraic construction}.

Let $T_{\Sigma}$ be a smooth projective toric variety over $\C$ corresponding to a fan $\Sigma$ in $\R^n$,
and $X \subset T_{\Sigma}$ a closed algebraic subvareity.

A semi-algebraic subset of  a $d$-dimensional complex algebraic variety $Y$ means
a semi-algebraic subset of $2d$-dimensional real algebraic variety $Y(\C)$.
We put $C_r^{\semialg}(Y(\C),\Q)$ the $\Q$-vector space of singular semi-algebraic chains, i.e., singular chains $\Delta^r \to Y(\C)$ which are  semi-algebraic.
When there are no confusions, we identify singular semi-algebraic chains with its image as formal sums of compact  oriented semi-algebraic subsets.
We sometimes identify them with their subdivisions.

\begin{dfn}\label{definition of admissible semi-algebraic subsets}
  A closed semi-algebraic set $S \subset T_{\Sigma}(\C)$ of pure dimension $r$
  is said to be admissible
  if and only if 
  we have 
  $$\dim(S  \cap   \overline{O(\sigma)(\C)}) \leq \dim S -2 \dim \sigma, $$ 
  for any $\sigma \in \Sigma$.
  We put 
  $AC_r(X(\C)):= AC_r(X(\C),\Q) $
  the subset of singular semi-algebraic chain $C_r^{\semialg} (X(\C),\Q)$ consisting of $\gamma$ such that the images of $\gamma $ and $\partial \gamma$ are admissible.
\end{dfn}

\begin{rem}\label{blow-up of admissible is also admissible}
  When $S\subset T_{\Sigma}$ is admissible,
  for any toric blow-up $T_{\Sigma'} \to T_{\Sigma}$ with smooth $T_{\Sigma'}$,
  the closure $\overline{S \cap \G_m^n(\C)} $ in $T_{\Sigma'}$ is also admissible.
\end{rem}

The following moving lemma is a  generalization of \cite[Theorem 2.8]{HKT20}, and it can be proved in the completely same way.
Remind that when there exists a subfan $\Lambda \subset \Sigma$ such that $X \subset T_{\Lambda}$ is Sch\"{o}n compact, the intersection $ X \cap O(\sigma)$ is of pure complex dimension $\dim \sigma$ for any $\sigma \in \Lambda$.
\begin{prp}[{Hanamura-Kimura-Terasoma, \cite[Theorem 2.8]{HKT20}}]\label{admissible moving lemma}
  When there exists a subfan $\Lambda \subset \Sigma$ such that  $X \subset T_{\Lambda}$ is Sch\"{o}n  compact,
  the natural map 
  $$ AC_r(X(\C),\Q) \to C_r^{\semialg} (X(\C),\Q)$$
  is a quasi-isomorphism, i.e., induces an isomorphism of homology groups.
\end{prp}

\begin{prp}[{Hanamura-Kimura-Terasoma, \cite[Theorem 4.4]{HKT20}}]\label{integration on admissible set is well-defined}
  Let $S \subset (\C^{\times})^n$ be a closed semi-algebraic set of pure dimension $r$
  such that the closure $\overline{S} \subset T_{\Sigma}(\C)$ is admissible.
  Then for any $f_i \in M$,
  the integration 
  $$\int_S d \log f_1 \wedge d \log f_2 \wedge \dots \wedge d \log f_r$$ absolutely converges.
\end{prp}
\begin{proof}
  By \cref{image of admissible is admissble}, 
  we may assume $r=n$.
  By subdividing $S$, we may assume $T_{\Sigma}=(\P^1)^n$.
  This case is just \cite[Theorem 4.4]{HKT20}.
\end{proof}

\begin{lem}\label{image of admissible is admissble}
  Let $S \subset (\C^{\times})^n$  be as in \cref{integration on admissible set is well-defined}
  and $\varphi \colon (\C^{\times})^n \to (\C^{\times})^r$ a torus morphism such that $\varphi(S)$ is of pure dimension $r$.
  Then 
  for some proper fan $\Sigma'$ in $(\C^{\times})^r$, the closure $\overline{\varphi(S)} \subset T_{\Sigma'}(\C)$ is admissible.
\end{lem}
\begin{proof}
  By \cref{blow-up of admissible is also admissible},
  we may assume that $\{\varphi(\sigma)\}_{\sigma \in \Sigma}$ is a fan. 
  (This $\varphi(\sigma)$ is the image of $\sigma$ under the linear map 
  $$\varphi \colon \Trop((\C^{\times})^n) \to \Trop((\C^{\times})^r)$$
  of Euclid spaces.)
  This is the required fan.
\end{proof}

\begin{rem}\label{remind face map in HKT}
  In \cite[Definition 3.8]{HKT20}, they defined the face map 
  $$ AC_r ((\P^1)^n(\C),D^n;\Q) \to AC_{r-2} ( (\P^1)^{n-1}(\C),D^{n-1};\Q)$$
  for a face $(\P^1)^{n-1}(\C) = \{0\} \times (\P^1)^{n-1}(\C)$ of $(\P^1)^{n}(\C)$.
  This is a generalization of the intersection products of closed subvarieties of $(\P^1)^n$ in generic position with $ \{0\} \times (\P^1)^{n-1}$.
  (For the definition of the boundaries $D^*$, see \cite{HKT20}. We do not need them here.)
  Let $T_{\Sigma'}$ be a smooth toric variety corresponding to a fan $\Sigma'$.
  For a $1$-dimensional cone $l \in \Sigma'$, in the same way, there is a map 
  $$ \partial_{\overline{O(l)}} \colon AC_r(T_{\Sigma'}(\C),\Q) \to AC_{r-2}(\overline{O(l)}(\C),\Q).$$
  We also call it the \emph{face map}.
  Here we define $AC_r(T_{\Sigma'}(\C),\Q)$ in the same way as those for smooth projective toric varieties,
  and $\overline{O(l)}(\C)$ is the closure in the toric variety $T_{\Sigma'} (\C)$.
  For simplicity,
  for a cone $ \sigma \in \Sigma'$ not containing  $l$ such that $\overline{O(\sigma)} \cap \overline{O(l)}$ is non-empty,
  we denote the face map 
  $$\partial_{\overline{O(l) }\cap  \overline{O(\sigma)} } \colon
   AC_{r}(\overline{O(\sigma)}(\C),\Q) \to AC_{r-2}(\overline{O(l)} \cap \overline{O(\sigma)}(\C),\Q)$$
  by $\partial_{\overline{O(l)}}$.
  Here the closures are taken in $T_{\Sigma'}(\C)$.
\end{rem}

\section{Weighted tropcializations of semi-algebraic chains}\label{Section Real semi-algebraic construction}
In this section, we shall prove that limits of integrals of differential forms given in \cref{section analytic} on admissble semi-algebraic subsets 
equal those of superforms on \textit{weighted tropicalizations}, which will be introduced in this section.
We also show that under a natural assumption, weighted tropicalizations give maps of tropical and the usual singular homology groups which are dual of $\Trop^*$.

Let $T_{\Sigma}$ be a smooth projective toric variety over $\C$ corresponding to a fan $\Sigma$ in $N_\R$,
and $X \subset T_{\Sigma}$ a smooth closed algebraic subvareity.

\subsection{Weighted tropicalizations}\label{subsection overview of semi-algebraic construction}
In this subsection, we will construct  $\wtTrop(V)$
for an admissble singular semi-algebraic chain $V \in C_r(X(\C),\Q)$.
We put $V_0:= V \cap \G_m^n$ a formal sum of closed semi-algebraic subsets. 
For a $1$-dimensional rational cone $l \subset \Trop(T_{\Sigma})$ 
such that $\relint (l) \subset N_\R$,
we put 
$v_l \in l \cap N_\R$
the primitive vector.

We shall construct an element $\wtTrop(V)$ in $  C_r^{\Trop} (\Trop(X),\C)$.
Let $P \in \Sigma$ be a cone. We fix an orientation of $P$.
Let $q:= \dim P$
and $l_1,\dots,l_q \in \Sigma$ be $1$-dimensional cones 
such that 
$\R_{\geq 0} \langle l_1,\dots,l_q\rangle = P,$ and
$v_{l_1} \wedge \dots \wedge v_{l_q} = 1_{\wedge^q \Span_\Z P}$.
We consider 
\begin{align}
& f_1 \wedge \dots \wedge f_{r-q} \\
\mapsto &
(-1)^{\frac{q(q-1)}{2}}
(f_1 \wedge \dots \wedge f_q, 1_{\wedge^q\Span_\Z P}) 
\int_{\partial_{\overline{O(P)}}(V) } 
-\frac{1}{2 \pi i}d \log f_{q+1} \wedge \dots \wedge -\frac{1}{2\pi i} d \log f_{r-q} \label{weighted tropicalization integration i}
\end{align}
for $f_i \in M$ with $f_i \in P^{\perp} \ (q+1 \leq i \leq r-q)$, 
where 
$(f_1 \wedge \dots \wedge f_q, 1_{\wedge^q\Span_\Z P}) $
is the pairing of $\wedge^q M$ and $\wedge^q N$
and 
$$
\partial_{\overline{O(P)}}(V) 
:=
\partial_{\overline{O(l_q)}}(\partial_{\overline{O(l_{q-1})}} (\dots \partial_{\overline{O(l_1)}}(V) \dots )).$$
This induces an element of $F_{r-q}(P,\Trop(X)) \otimes \C$.
We denote it by 
$\wtTrop(V)_P$.
We call it the \emph{weight} of $\Trop(V)$ at $P$.
We put 
$$\wtTrop(V):= \sum_{P \in \Sigma} \wtTrop(V)_P \otimes [P] \in C_r^{\Trop} (\overline{\Sigma},\C),$$
where $[P] \in C_q^{\sing}(P,\Z)$ is given by the fixed orientation,
and the fan $\overline{\Sigma}$ in $\Trop(T_{\Sigma})$ consists of cones $\overline{\sigma} \cap \overline{N_{\tau,\R}}$ $(\sigma,\tau \in \Sigma)$ (where the closures are taken in  $\Trop(T_\Sigma))$.
(Here we do not care about the difference between $[P]$ and a homeomorphism $\Delta^q \cong P$.)
We call $\wtTrop(V)$ the \emph{weighted tropicalization} of $V$.
We often identify $\wtTrop(V)$ and its subdivisions (i.e., subdivisions of singular chain parts).
Then we have $\wtTrop(V) \in C_r^{\Trop}(\Lambda,\C)$
for a fan structure $\Lambda$ of $\Trop(X)$.
By abuse of notation, we write simply
$$\wtTrop(V) \in C_r^{\Trop}(\Trop(X),\C).$$

The weighted tropicalization $\wtTrop(V)$  is independent of the toric varieties $T_{\Sigma}$ in the following sense.
\begin{lem}\label{weighted tropicalizations independent of the choice of fan}
 Let $\psi \colon T_{\Sigma'} \to T_{\Sigma} $ be a toric blow-up from a smooth toric variety $T_{\Sigma'}$.
 We put $V':= \overline{V_0}$ the closure in $T_{\Sigma'}(\C)$.
 Then
 we have  
 $$\psi_*(\wtTrop(V'))= \wtTrop(V)  \in C_r^{\Trop} (\Trop(X),\C)  .$$
Here $\psi_*$ is the natural map of tropical chains induced from the map 
 $\psi \colon \Trop(T_{\Sigma'}) \to \Trop(T_{\Sigma}) $ of tropical toric varieties.
\end{lem}
\begin{proof}
  We may assume that $\psi \colon T_{\Sigma'} \to T_{\Sigma}$ is the blow-up at smooth toric center $\overline{O(\sigma)}$ $(\sigma \in \Sigma)$.
  Let $\tau \in \Sigma$ and  $\tau' \in \Sigma'$ be cones such that 
  $\sigma \subset \tau$ and
   $\Sigma' \ni \tau' \to \tau \in \Sigma$ (i.e., $\relint (\tau' ) \subset \relint(\tau)$).
  When $\dim \tau' \leq \dim \tau -1$,
  then since $V' \cap O(\tau')$ is contained in the inverse image of $V \cap O(\tau)$ 
  (which is the trivial $(\C^\times)^{\dim \tau -\dim \tau'}$-bundle), 
  the $\tau'$-component $\wtTrop(V')_{\tau'}$ is $0$.
  We assume $\dim \tau' = \dim \tau$.
  It suffices to show that
  $
  \psi \colon \partial_{\overline{O(\tau')}}(V')  \to
   \partial_{\overline{O(\tau)}}(V) $
  is a bijection (including multiplicities).
  To prove this, as in the proof of \cite[Proposition 3.10]{HKT20}, it suffices to show that 
  $$\psi_* ( [\overline{O(\tau')}] \cap [V'])= [\overline{O(\tau)}] \cap [V] $$
  in $H_{\dim V -2 \dim \tau}^{\sing}(V \cap \overline{O(\tau)} ,\partial V \cap \overline{O(\tau)},\Q)$.
  This is just projection formula.
  (See the proof of \cite[Proposition 3.10]{HKT20}.)
\end{proof}

Alessandrini \cite{Ale13} introduced the tropcialization $\Trop(V \cap \G_m^n(\C)) \subset N_\R$ of a semi-algebraic subset $V$ and showed that it is a finite union of rational polyhedral cones.
Supports of our weighted tropicalizations are contained in (the closure of) his tropicalizations. This follows from \cref{weighted tropicalizations independent of the choice of fan}.

\begin{rem}
For a closed complex algebraic subvariety $Y \in AC_{2p}(X(\C),\Q)$ of dimension $p$,
the classical weighted tropicalization $\wtTrop(Y)$ (see \cref{remark weight tropicalization of complex algebraic varieties}) coincides with the one defined in this section.
To see this, by \cite[Theorem 1.4]{LQ11} and \cref{weighted tropicalizations independent of the choice of fan}, we can reduce the assertion to the case that 
$\Lambda \subset \Sigma$ 
and
$Y \subset T_{\Lambda}$ is Sch\"{o}n compactification. 
Then 
this follows from  the fact that the classical weight $m_P$ (\cref{remark weight tropicalization of complex algebraic varieties}) is the intersection number $Y \cdot \overline{P}$ (\cite[Theorem 6.7.7]{MS15}). 
\end{rem}

\subsection{Comparison of integrals}\label{section comparison of maps}
We shall prove the equality of tropical integrals and limits of archimedean integrals.
 Let $V \in C_{p+q}(X,\C)$ be an admissible singular semi-algebraic chain 
 and $w \in \mathscr{A}_{\Trop}^{p,q}(\Trop(T_{\Sigma}))$.
For $\epsilon > 0$,
the $(p+q)$-form
$ -\epsilon \log \lvert \cdot \rvert^*(w) $
is a compactly-supported smooth $(p+q)$-form defined on $X(\C)$.
Hence by \cite[Theorem2.6]{HKT15},
the integral 
$$
 \int_V 
 -\epsilon \log \lvert \cdot \rvert^*(w) 
$$
is well-defined.
(See \cite[Section 1]{HKT15} for definitions of this integration, based on the Lebesgue integration.)

\begin{prp}\label{comparison of maps analytic and semi-algebraic}
 We have 
 $$ \int_{\wtTrop ( V )} w =
 \lim_{\epsilon \to 0} 
 \int_V 
 -\epsilon \log \lvert \cdot \rvert^*(w). 
 $$
\end{prp}

\begin{proof}
 We may and do assume the following.
  \begin{itemize}
    \item $X = T_{\Sigma}=\P^n$, and $V \subset \A^n \subset \P^n$.
    \item  We denote the coordinate of $\A^s \times \G_m^{n-s} (\C)$ by $z_1, \dots z_n$ and that of $\Trop(\A^s \times \G_m^{n-s})$ by $x_1,\dots,x_n$.
    \item 
    \begin{align*}
    w = f(x)  
    & d'x_{I_0}  \wedge d' x_{I_1} 
    \otimes
    d'' x_{I_0}  \wedge    d'' x_{I_2}
    \end{align*} 
    ($I_0  =\{1,\dots,r\}  \ (r \leq s),   I_j \subset \{r+1,\dots,n\} \ (j=1,2), 
     I_1 \cap I_2 = \emptyset$).
  \end{itemize}
  We put $V_0:= V \cap \G_m^n(\C)$.
Note that
$$ d (- \frac{\epsilon}{2} \log \overline{z})  \wedge
\frac{1}{2 \pi i} d (- \log z)
= d(-\epsilon \log r) \wedge \frac{1}{2\pi} d (-\theta)$$
for  $z=r e^{i \theta}$.
Hence 
\begin{align*}
 &-\epsilon \log \lvert \cdot \rvert^*(w)\\
 =&  f(-\epsilon \log \lvert z_1 \rvert , \dots, -\epsilon \log \lvert z_n \rvert)  
  d (-\frac{\epsilon}{2} \log \overline{z_{I_0}} )
  \wedge d (-\frac{\epsilon}{2} \log \overline{z_{I_2}})
  \wedge 
  d (-\frac{1}{2\pi i} \log z_{I_0})  
  \wedge 
  d (-\frac{1}{2\pi i} \log z_{I_1})
  \\
 =& (-1)^{ p \lvert I_2 \rvert }  f(x_1,\dots,x_r, -\epsilon \log \lvert z_{r+1} \rvert , \dots, -\epsilon \log \lvert z_n \rvert)  
  d x_{I_0} \wedge d (-\frac{1}{2 \pi} \theta_{I_0})
  \wedge   d (-\frac{1}{2\pi i} \log z_{I_1}) \wedge    d (-\frac{\epsilon}{2} \log \overline{z_{I_2}}).
\end{align*}
Its ``limit'' under $\epsilon \mapsto 0$ is 
 $$ (-1)^{p \lvert I_2 \rvert } f(x_1,\dots,x_r, 0, \dots,0)  
  d x_{I_0} \wedge d (-\frac{1}{2\pi} \theta_{I_0})
  \wedge   d (-\frac{1}{2\pi i} \log z_{I_1}) \wedge    0^{\lvert I_2 \rvert }.$$
 The integration of this ``limit'' on 
 $$[0,\infty]^r \times [\overline{V_0} \cap (O(\sigma_r)(\C) \times (S^1)^r) ]$$
 is the integration $ \int_{\wtTrop(V)} w $.
 Hence it is enough to justify this ``limit'' and prove compatibility of the integration and this ``limit''. 
 This will be done by a usual argument of the Lebesgue integration.

 First, we justify the ``limit''.
For $J_1 \subset I_1$ and $J_2 \subset I_2$,
we put 
\begin{align*}
 &-\epsilon \log \lvert \cdot \rvert^*(w)_{J_1,J_2} \\
 :=& (-1)^{p \lvert I_2 \rvert } f(x_1,\dots,x_r,  -\epsilon \log \lvert z_{r+1} \rvert , \dots, -\epsilon \log \lvert z_n \rvert)  
  d x_{I_0} \wedge d(- \frac{1}{ 2\pi} \theta_{I_0}) \\
 & \wedge   d (-\frac{1}{2\pi i} \log r_{J_1})
  \wedge   d (-\frac{1}{2\pi} \theta_{I_1 \setminus J_1})
  \wedge    d (-\frac{\epsilon}{2} \log r_{J_2})
  \wedge     d ( \frac{i \epsilon}{2} \theta_{I_2 \setminus J_2}).
\end{align*}
We have 
$$
 -\epsilon \log \lvert \cdot \rvert^*(w)
= \sum_{J_1 \subset I_1,\ J_2 \subset I_2} 
-\epsilon \log \lvert \cdot \rvert^*(w)_{J_1,J_2}. $$
We put 
\begin{align*}
\varphi_{J_1,J_2,\epsilon} \colon & (\C^\times)^r \times (\C^\times)^{J_1 \cup J_2} \times (\C^\times)^{I_1 \cup I_2 \setminus (J_1 \cup J_2)} \times (\C^\times)^{\{r+1,\dots,n\} \setminus (I_1 \cup I_2)}  \\
& \to  \R^r \times [0,2 \pi]^r \times \R_{>0}^{J_1 \cup J_2} \times [0,2 \pi]^{I_1 \cup I_2 \setminus (J_1 \cup J_2)}
\end{align*}
the composition of the projection 
$$(\C^\times)^r \times (\C^\times)^{J_1 \cup J_2} \times (\C^\times)^{I_1 \cup I_2 \setminus (J_1 \cup J_2)} \times (\C^\times)^{\{r+1,\dots,n\} \setminus (I_1 \cup I_2)}  
\to (\C^\times)^r \times (\C^\times)^{J_1 \cup J_2} \times (\C^\times)^{I_1 \cup I_2 \setminus (J_1 \cup J_2)} $$
and
the product of 
$$(\C^\times)^r \ni (z_i = r_i e^{i \theta_i})_i \mapsto (x_i:= -\epsilon \log r_i,\theta_i)_i \in \R^r \times [0,2 \pi]^r,$$
$$(\C^\times)^{J_1 \cup J_2} \ni (z_i = r_i e^{i \theta_i})_i \mapsto (r_i)_i \in \R_{>0}^{J_1 \cup J_2},$$
$$(\C^\times)^{I_1 \cup I_2 \setminus (J_1 \cup J_2)} \ni (z_i = r_i e^{i \theta_i})_i \mapsto (\theta_i)_i \in  [0,2 \pi]^{I_1 \cup I_2 \setminus (J_1 \cup J_2)}.$$
We put $(\varphi_{J_1,J_2,\epsilon})_*([V_0])$ the push-forward of $V_0$ as a formal sum of $r$-dimensional closed oriented semi-algebraic subsets 
in the same way as the push-forward of singular semi-algebraic chains. (Note that $V_0$ and $(\varphi_{J_1,J_2,\epsilon})_*([V_0])$ are in general not compact.) 
For each semi-algebraic set $ U \subset \varphi_{J_1,J_2,\epsilon} (V_0)$ such that 
the inverse image $\varphi_{J_1,J_2,\epsilon}^{-1}(U) \cap V_0$ is a disjoint union of finitely many semi-algebraic sets $U_s$ which are homeomorphic to $U$,
we put 
$$f_{\epsilon}(a):= \sum_s f(a_1,\dots,a_r,
-\epsilon \log \lvert z_{r +1} \rvert_s (e^{-\frac{a_i}{\epsilon}},b_i,c_j,d_k) , \dots, 
-\epsilon \log \lvert z_{n} \rvert_s (e^{-\frac{a_i}{\epsilon}},b_i,c_j,d_k) )$$
for $a=(a_i,b_i,c_j,d_k) \in U$,
where we denote the restriction of 
$ -\epsilon \log \lvert z_{l} \rvert $
on $U_s$ by
$ -\epsilon \log \lvert z_{l} \rvert_s$,
and
we denote it
as a function on $\varphi_{J_1,J_2,\epsilon}(U_s)$ by 
$ -\epsilon \log \lvert z_{l} \rvert_s (e^{-\frac{x_i}{\epsilon}},\theta_i,r_j,\theta_k) $.
We define a function $f_{\epsilon}$ on 
$$\R^r \times [0,2 \pi]^r \times \R_{>0}^{J_1 \cup J_2} \times [0,2 \pi]^{I_1 \cup I_2 \setminus (J_1 \cup J_2)}$$
by putting $f_{\epsilon}(a)=0$ for  $a $ such that $\varphi_{J_1,J_2,\epsilon}^{-1}(a) \cap V_0$ is empty or an infinite set.
Then we have
\begin{align*}
 &\int_V -\epsilon \log \lvert \cdot \rvert^*(w)_{J_1,J_2} \\
  =&\int_{
  \R^r \times [0,2 \pi]^r \times \R_{>0}^{J_1 \cup J_2} \times [0,2 \pi]^{I_1 \cup I_2 \setminus (J_1 \cup J_2)} }
  (-1)^{p \lvert I_2 \rvert}
  \frac{1}{r^{J_1}} \cdot \frac{1}{r^{J_2}} \cdot  
 f_\epsilon  \\
  & d x_{I_0} \wedge d(-\frac{1}{ 2\pi}  \theta_{I_0}) 
  \wedge   d (-\frac{1}{2\pi i}  r_{J_1})
  \wedge   d (-\frac{1}{2\pi} \theta_{I_1 \setminus J_1})
  \wedge    d (-\frac{\epsilon}{2}  r_{J_2})
  \wedge     d ( \frac{i \epsilon}{2} \theta_{I_2 \setminus J_2}).
\end{align*}
For 
$$a= (a_i,b_i,c_j,d_k) \in 
[0,\infty)^r \times \pr(\overline{V_0} \cap (O(\sigma_r)(\C) \times (S^1)^r) ) \subset 
\R^r \times [0,2 \pi]^r \times \R_{>0}^{J_1 \cup J_2} \times [0,2 \pi]^{I_1 \cup I_2 \setminus (J_1 \cup J_2)}$$
and $U_s$ such that 
$\varphi_{J_1,J_2,\epsilon}(U_s)$ contains $a$ for any small $\epsilon$,
the limit
$$\lim_{\epsilon \to 0} \log \lvert z_l \rvert_s (e^{\frac{- a_i}{\epsilon}}, b_i,c_j,d_k) \ ( r+1 \leq l)$$
exists and takes a finite value.
Hence for almost all such a $a$, 
we have 
$$\lim_{\epsilon \to 0} f_{\epsilon}(a) 
= 1_{ \pr_*([\overline{V_0} \cap (O(\sigma_r)(\C) \times (S^1)^r)] )} (b_i,c_j,d_k) f(a_1,\dots,a_r,0,\dots,0),$$
where we put
$$ \pr \colon O(\sigma_r)(\C) \times (S^1)^r
\to \R_{>0}^{J_1 \cup J_2} \times [0, 2 \pi]^{I_1 \cup I_2 \setminus (J_1 \cup J_2)} \times [0,2\pi]^r$$
the map given by 
$z_j=r_j e^{\theta_j} \mapsto r_j \ (j \in J_1 \cup J_2)$
and
$z_k=r_k e^{\theta_k} \mapsto \theta_k \ (k \in I_1 \cup I_2 \setminus (J_1 \cup J_2))$,
and
we put  $\pr_*([\overline{V_0} \cap (O(\sigma_r)(\C) \times (S^1)^r)])$ 
the push-forward of the semi-algebraic chain 
$[\overline{V_0} \cap (O(\sigma_r)(\C) \times (S^1)^r)]$
and we put  
$$1_{ \pr_*([\overline{V_0} \cap (O(\sigma_r)(\C) \times (S^1)^r)] )}$$  an integrable function on 
$$
 [0,2 \pi]^r \times \R_{>0}^{J_1 \cup J_2} \times [0,2 \pi]^{I_1 \cup I_2 \setminus (J_1 \cup J_2)}$$
such that 
for any integrable $(p+q)$-form $g$, 
we have 
$$\int_{
 [0,2 \pi]^r \times \R_{>0}^{J_1 \cup J_2} \times [0,2 \pi]^{I_1 \cup I_2 \setminus (J_1 \cup J_2)}}
1_{ \pr_*([\overline{V_0} \cap (O(\sigma_r)(\C) \times (S^1)^r)] )} \cdot g 
=  \int_{ [\overline{V_0} \cap (O(\sigma_r)(\C) \times (S^1)^r)]} \pr^* g. $$
Consequently, we have 
\begin{align*}
   & \sum_{J_1,J_2}\int_{
  \R^r \times [0,2 \pi]^r \times \R_{>0}^{J_1 \cup J_2} \times [0,2 \pi]^{I_1 \cup I_2 \setminus (J_1 \cup J_2)} }
  \lim_{\epsilon \to 0}
  (-1)^{p \lvert I_2 \rvert}
  \frac{1}{r^{J_1}} \cdot \frac{1}{r^{J_2}} \cdot  
 f_\epsilon  \\
  & d x_{I_0} \wedge d(-\frac{1}{ 2\pi}  \theta_{I_0}) 
  \wedge   d (- \frac{1}{2\pi i} r_{J_1})
  \wedge   d (-\frac{1}{2\pi} \theta_{I_1 \setminus J_1})
  \wedge    d (-\frac{\epsilon}{2}  r_{J_2})
  \wedge     d ( \frac{i \epsilon}{2} \theta_{I_2 \setminus J_2})\\
  =& 
\int_{\wtTrop (V)}w
\end{align*}

Next, we will show that the integration is compatible with limit.
To see this it is enough to show that  
  $$\frac{1}{r^{J_1 \cup J_2}}\lvert 
  f_\epsilon   \rvert
   $$
   is bounded above by an integrable function on
$$ \R^r \times [0,2 \pi]^r \times \R_{>0}^{J_1 \cup J_2} \times [0,2 \pi]^{I_1 \cup I_2 \setminus (J_1 \cup J_2)}.$$
By the boundary condition of $w$ (\cref{def:forms}), 
there exists an $\R$-valued integrable function  $h$ on $\R^r$ 
such that $\pr_{\R^n \to \R^r}^*h \geq \lvert f \rvert $ on $\R^n$,
where $\pr_{\R^n \to \R^r}$ is the projection to the first $r$ coordinates.
Then we have
$$ C \cdot 1_{\R^r \times \psi_{J_1,J_2}(\varphi_{J_1,J_2,\epsilon}(V_0))} \cdot \pr^*_{
\R^r \times [0,2 \pi]^r \times \R_{>0}^{J_1 \cup J_2} \times [0,2 \pi]^{I_1 \cup I_2 \setminus (J_1 \cup J_2)} \to \R^r}
h \geq \lvert f_{\epsilon} \rvert$$ 
on 
$$\R^r \times [0,2 \pi]^r \times \R_{>0}^{J_1 \cup J_2} \times [0,2 \pi]^{I_1 \cup I_2 \setminus (J_1 \cup J_2)} ,$$
where 
$C \in \R$ is a constant,
and
we put 
$$ \psi_{J_1,J_2} \colon 
\R^r \times [0,2 \pi]^r \times \R_{>0}^{J_1 \cup J_2} \times [0,2 \pi]^{I_1 \cup I_2 \setminus (J_1 \cup J_2)}
 \to  [0,2 \pi]^r \times \R_{>0}^{J_1 \cup J_2} \times [0,2 \pi]^{I_1 \cup I_2 \setminus (J_1 \cup J_2)}$$
the projection.
Note that $$\R^r \times \psi_{J_1,J_2}(\varphi_{J_1,J_2,\epsilon}(V_0))$$ is independent of $\epsilon $.
We have
\begin{align*}
&\int_{\R^r \times [0,2 \pi]^r \times \R_{>0}^{J_1 \cup J_2} \times [0,2 \pi]^{I_1 \cup I_2 \setminus (J_1 \cup J_2)}}
 \frac{1}{r^{J_1 \cup J_2}} \cdot 1_{\R^r \times \psi_{J_1,J_2}(\varphi_{J_1,J_2,\epsilon}(V_0))} \\
& \cdot  \pr^*_{
\R^r \times [0,2 \pi]^r \times \R_{>0}^{J_1 \cup J_2} \times [0,2 \pi]^{I_1 \cup I_2 \setminus (J_1 \cup J_2)} \to \R^r} (h) d \mu \\
=& \int_{\R^r}h d \mu' \int_{\psi_{J_1,J_2}(\varphi_{J_1,J_2,\epsilon}(V_0))} \frac{1}{r^{J_1 \cup J_2}}d \mu'',
\end{align*}
where $d \mu , d \mu', d \mu''$ are the Lebesgue  measures on $\R^{p+q}$, $\R^r$, and $\R^{p+q-r}$, respectively.
Since $h$ is integrable and $V$ is admissible, by \cite[Theorem 4.4]{HKT20} (\cref{integration on admissible set is well-defined}),
this integration takes finite value.
Hence by a limit argument of the Lebesgue integration,
we have 
\begin{align*}
  & \lim_{\epsilon \to 0}  \int_V -\epsilon \log \lvert \cdot \rvert^*(w) \\
  =& \sum_{J_1,J_2}\lim_{\epsilon \to 0}  \int_V -\epsilon \log \lvert \cdot \rvert^*(w)_{J_1,J_2} \\
  = &\sum_{J_1,J_2}  \int_{ (x_i,\theta_i,r_j,\theta_k) \in \R^r \times [0,2 \pi]^r \times \R_{>0}^{J_1 \cup J_2} \times [0,2 \pi]^{I_1 \cup I_2 \setminus (J_1 \cup J_2)}}\\
  & \lim_{\epsilon \to 0}
  (-1)^{p \lvert I_2 \rvert }
  \frac{1}{r^{J_1 \cup J_2}}
  f_{\epsilon} (x_1,\dots,x_r, -\epsilon \log \lvert z_{r+1} \rvert
  (e^{\frac{- x_i}{\epsilon}}, \theta_i,r_j,\theta_k), \dots,-\epsilon \log \lvert z_n \rvert (e^{\frac{x_i}{\epsilon}}, \theta_i,r_j,\theta_k) ) \\ 
  & d x_{I_0} \wedge d (-\frac{1}{ \pi} \theta_{I_0}) 
  \wedge   d (-\frac{1}{2\pi i}  r_{J_1})
  \wedge   d (-\frac{1}{2\pi} \theta_{I_1 \setminus J_1})
  \wedge    d (-\frac{\epsilon}{2}  r_{J_2})
  \wedge    d ( \frac{i\epsilon}{2} \theta_{I_2 \setminus J_2}) \\
  =& \int_{\wtTrop(V)} w
\end{align*}
\end{proof}

\subsection{Induced maps of homology}\label{subsection rationality}
We study the maps of homology groups induced by weighted tropicalizations.
\begin{prp}\label{compatibility of phase tropicalization and boundary maps}
  The map $\wtTrop$ is compatible with the boundary maps, i.e., 
  $$\wtTrop \circ \partial = \partial \circ \wtTrop.$$
\end{prp}
\begin{proof}
  This follows from the compatibility of the face maps and the usual boundary map \cite[Proposition 3.6(1)]{HKT20} and the generalized Cauchy formula \cite[Theorem 4.3]{HKT20}.
\end{proof}

By \cref{compatibility of phase tropicalization and boundary maps},
the map 
$$\wtTrop \colon AC_r(X(\C),\Q)   \to C_r^{\Trop}(\Trop(X),\C)/(\text{subdivisions})$$
induces a map 
$$\wtTrop \colon H_r(AC_*(X(\C),\Q))  \to H_r^{\Trop}(\Trop(X),\C).$$
of homology groups.

\begin{prp}\label{rationality of map}
  Let $V \in AC_r(X(\C),\Q)$ such that $\partial(V)=0$.
  Then we have
  $$\wtTrop(V) \in C_r(\Trop(X),\Q).$$
\end{prp}
\begin{proof}
  By compatibility of the boundary map and the face maps \cite[Proposition 3.6 (1)]{HKT20} and rationality of the face maps (\cref{remind face map in HKT}),
  the assertion follows from \cref{rationality of logarithmic integration}.
\end{proof}
\begin{lem}\label{rationality of logarithmic integration}
 Let $V \in AC_r(X(\C),\Q)$ such that $\partial(V)=0$. 
 Then we have
  $$\int_{V} -\frac{1}{2\pi i} \frac{df_1}{f_1} \wedge \dots \wedge -\frac{1}{2\pi i} \frac{df_r}{f_r} \in \Q  \ (f_i \in M).$$
\end{lem}
\begin{proof}
 We prove by induction on $r$.
 When $r$ is $0$, the assertion is trivial.
 When $r$ is $1$, by admissibility, the cycle $V$ is contained in $\G_m^n(\C)$.
 Since $-\frac{1}{2\pi i} \frac{df}{f} \in H^1_{\sing}(\G_m^n(\C), \Q)$, the assertion holds in this case.
 We assume $r \geq 2$.
 By \cref{admissible moving lemma},
 since $H_{\sing}^*(T_{\Sigma}(\C),\Z)$ is isomorphic to the Chow group of $T_{\Sigma}$,
 there exist an algebraic cycle $W \in AC_r(T_{\Sigma}(\C),\Q)$ (i.e., a finite sum of closed complex algebraic subvarieties)
 and
 a semi-algebraic chain $Z \in AC_r(T_{\Sigma}(\C),\Q)$
 such that 
 $$V-W=\partial Z \in AC_r(T_{\Sigma}(\C),\Q).$$
 Since 
 $W$ is a complex algebraic cycle, we have
 $$\int_{W} -\frac{1}{2\pi i} \frac{df_1}{f_1} \wedge \dots \wedge -\frac{1}{2\pi i} \frac{df_r}{f_r} =0  \ (f_i \in M).$$
 By the generalized Cauchy formula \cite[Threorem 4.3]{HKT20} and the hypothesis of the induction,
 we have
 $$\int_{\partial Z} -\frac{1}{2\pi i} \frac{df_1}{f_1} \wedge \dots \wedge -\frac{1}{2\pi i} \frac{df_r}{f_r}
 \in \Q  \ (f_i \in M).$$
 (Note that the notation of \cite{HKT20} and ours are different. 
 The generalized Cauchy formula \cite[Theorem 4.3]{HKT20} also holds for smooth proper toric varieties since the assertion is local.)
 Consequently,  we have
 $$\int_{V} -\frac{1}{2\pi i} \frac{df_1}{f_1} \wedge \dots \wedge -\frac{1}{2\pi i} \frac{df_r}{f_r} \in \Q  \ (f_i \in M).$$
\end{proof}

By \cref{rationality of map}, we get a map 
$$\wtTrop \colon H_r(AC_*(X(\C),\Q)) \to H_r^{\Trop}(\Trop(X),\Q).$$

When $X \subset T_{\Sigma}$ is Sch\"{o}n,
by moving lemma (\cref{admissible moving lemma} (\cite[Theorem 2.8]{HKT20})) and existence of semi-algebraic triangulations of $X(\C)$, we get a map 
$$ \wtTrop \colon H_r^{\sing}(X(\C),\Q))  \to H_r^{\Trop}(\Trop(X),\Q).$$

Consequently, by \cref{comparison of maps analytic and semi-algebraic} and the tropical de Rham theorem \cite[Theorem 1]{JSS17}, we have the following.
\begin{cor}\label{comparison of Trop^* and wtTrop2}
 We assume that 
 $X \subset T_{\Sigma}$ is Sch\"{o}n and 
 every class in the tropical Dolbeault cohomology group $H_{\Trop}^{p,q}(\Trop(X))$ has a $d'$-closed representative.
Then we have 
$$\Trop^* =  \wtTrop^\vee \colon 
H_{\Trop}^{p,q}(\Trop(X),\R) \to H_{\sing}^{p+q}(X,\C).$$
  In particular, for $p \neq q$, the map
  $$ \Trop^* \colon H_{\Trop,\Dol}^{p,q}(\Trop(X)) \to H_{\sing}^{p+q}(X,\C)$$
  is $0$-map.
\end{cor}

\begin{cor}\label{comparison of Trop^* and wtTrop}
  We assume that for any smooth toric variety $T_{\Lambda}$ and  Sch\"{o}n compact subvariety $X' \subset T_{\Lambda}$, 
  every class in the tropical Dolbeault cohomology group $H_{\Trop}^{p,q}(\Trop(X'))$ has a $d'$-closed representative.
  Then for $p\neq q $ and a smooth projective variety $Y$ over $\C$,
  the map
  $$ \Trop^* \colon H_{\Trop,\Dol}^{p,q}(Y) \to H^{p,q}(Y(\C))$$
  is $0$-map.
\end{cor}
\begin{proof}
This follows from the existence of Sch\"{o}n very affine open subvarieties \cite[Theorem 1.4]{LQ11} 
and injectivity of the pull-back maps of tropical cohomology groups under blow-ups of smooth algebaic varieties over trivially valued fields at smooth center (which follows from \cite[Theorem 1.3]{M20-2} and \cite[Lemma 12.9]{Ros96}).
\end{proof}

\begin{rem}
 The assumption of the exsitence of $d'$-closed representatives is  natural as an analog of Dolbeault cohomology of compact K\"{a}hler manifolds. (Every class of them has a unique harmonic representative.)
 However, the correctness of this assumption is still not known for most cases.
\end{rem}

\subsection*{Acknowledgements}
Many parts of this paper were studied when I was a doctorial student.
I would like to express my sincere thanks to my adviser Tetsushi Ito for  invaluable advice and persistent support.
The remaining parts were studied when I am a postdoctoral researcher at Academia Sinica.  
I also deeply thank my mentor Yuan-Ping Lee for his kind encouragement and helpful advice.
I thank Klaus K\"{u}nnemann for letting me know Bloch-Kato and Gabber's result (\cite[Theorem 2.1]{BK86}).
This work was supported by JSPS KAKENHI Grant Number18J21577.
I thank Helge Ruddat for telling his related results.

\end{document}